\documentclass[12pt]{article}
\usepackage[T1]{fontenc}
\usepackage{amsmath}
\usepackage{amssymb}
\usepackage{amsthm}
\usepackage{tikz}

\newtheorem{theorem}{Theorem}[section]
\newtheorem{definition}[theorem]{Definition}
\newtheorem{lemma}[theorem]{Lemma}

\newtheorem{proposition}[theorem]{Proposition}
\newtheorem{corollary}[theorem]{Corollary}
\newtheorem{sublemma}[theorem]{Sublemma}

\def\fl#1{\smash{\mathop{\hbox to 12mm{ \rightarrowfill\ }}\limits^{\textstyle #1}}}

\newcommand{\lhp}{\textrm{\raisebox{-0.5ex}{\huge (}}}
\newcommand{\lbp}{\textrm{\raisebox{-0.2ex}{\Large (}}}
\newcommand{\rhp}{\textrm{\raisebox{-0.5ex}{\huge )}}}
\newcommand{\rbp}{\textrm{\raisebox{-0.2ex}{\Large )}}}

\newcommand{\mi}{{\underline{i}}}
\newcommand{\mk}{{\underline{k}}}
\newcommand{\ml}{{\underline{\ell}}}
\newcommand{\mj}{{\underline{j}}}
\newcommand{\me}{{\underline{\varepsilon}}}
\newcommand{\meta}{{\underline{\eta}}}
\newcommand{\mnu}{{\underline{\nu}}}
\newcommand{\mtau}{{\underline{\tau}}}
\newcommand{\mo}{{\underline{0}}}

\newcommand{\cs}[1]{\textrm{{\Large $\sharp$}}_{#1}}

\newcommand{\proofof}[1]{\noindent{\textit{Proof of #1.}}}
\newcommand{\fin}{\hfill$\square$\bigskip}

\title{Finite type invariants of rational homology 3-spheres}
\author{Delphine Moussard}
\date{ }

\begin{document}

\maketitle

\begin{abstract}
We consider the rational vector space generated by all rational homology spheres up to orientation-preserving homeomorphism, 
and the filtration defined on this space by Lagrangian-preserving rational homology handlebody replacements. We identify the graded space 
associated with this filtration with a graded space of augmented Jacobi diagrams.


\end{abstract}

\tableofcontents

    \section{Introduction}

    \subsection{Finite type invariants}

The greatest achievements in the theories of finite type invariants are theorems that express the graded spaces 
associated with topological filtrations of vector spaces generated by knots or manifolds as combinatorial vector spaces
generated by Feynman diagrams. The two main examples of these theorems, that are useful to classify invariants 
and to evaluate their power, concern the Vassiliev filtration of the 
space generated by the knots in $S^3$, and the Goussarov-Habiro filtration of the space generated by the integral
homology 3-spheres ($\mathbb{Z}$HS), that are oriented compact 3-manifolds with the same integral homology as $S^3$. 
The graded space associated with the Vassiliev filtration was identified with a space of Jacobi diagrams by an isomorphism 
induced by the Kontsevich integral (see \cite{Kon} and the Bar-Natan article \cite{BN}). Several filtrations of the space generated 
by the $\mathbb{Z}$HS's were defined. In \cite{GGP}, Garoufalidis, Goussarov and Polyak compared various 
filtrations, and defined a surjective map from a graded space of Jacobi diagrams to the graded space associated 
with the Goussarov-Habiro filtration. In \cite{Le}, Le proved that this map is an isomorphism by showing that the LMO invariant
that he constructed in \cite{LMO} with the help of Murakami and Ohtsuki is a universal finite type invariant of $\mathbb{Z}$HS's.
In \cite{AL}, Auclair and Lescop defined the Goussarov-Habiro filtration and the properties of the graded space,
algebraically, using Lagrangian-preserving integral homology handlebody replacements.

In this article, we will consider the rational vector 
space generated by all the rational homology spheres ($\mathbb{Q}$HS), 
that are the oriented compact 3-manifolds with the same rational homology as $S^3$. 
We will define a filtration on this space by means of LP-surgeries, that are Lagrangian-preserving rational homology handlebody replacements.
Our main result (Theorem \ref{thprinc}) identifies the graded space associated with this filtration with a graded space of diagrams.
The role of the LMO invariant in the integral case will be held here by the KKT invariant of rational homology spheres
constructed by Kontsevich, and proved to be a universal finite type invariant of $\mathbb{Z}$HS's by Kuperberg and Thurston in \cite{KT}. 
Lescop has proved in \cite{Les} that the KKT invariant $Z_{KKT}=(Z_{n,KKT})_{n\in\mathbb{N}}$ satisfies a universality property with respect 
to LP-surgeries. Massuyeau has proved in \cite{Mas} that the LMO invariant $Z_{LMO}=(Z_{n,LMO})_{n\in\mathbb{N}}$ satisfies the same property. 
As we prove at the end of Section \ref{secfil}, these results and our main theorem imply that $Z_{LMO}$ and $Z_{KKT}$ are equivalent 
in the following sense:
\begin{theorem} \label{th+}
 Let $M$ and $N$ be $\mathbb{Q}$HS's such that $|H_1(M;\mathbb{Z})|=|H_1(N;\mathbb{Z})|$, where $|.|$ denotes the cardinality. 
Then, for any $n\in\mathbb{N}$:
$$\lbp Z_{k,LMO}(M)=Z_{k,LMO}(N)\textrm{ for all } k\leq n \rbp \Leftrightarrow \lbp Z_{k,KKT}(M)=Z_{k,KKT}(N)\textrm{ for all } k\leq n \rbp .$$
\end{theorem}

    \subsection{The Goussarov-Habiro filtration} \label{subsecGH}

Throughout the article, the manifolds will be compact, connected, and oriented.
When it does not seem to cause confusion, we will use the same notation for a curve and its homology class.

The {\em standard Y-graph} is the graph $\Gamma_0\subset \mathbb{R}^2$ represented in Figure \ref{figY}. 
With $\Gamma_0$ is associated a regular neighborhood $\Sigma(\Gamma_0)$ of $\Gamma_0$ in the plane. 
\begin{figure}[htb] 
\begin{center}
\begin{tikzpicture} [scale=0.15]
\newcommand{\feuille}[1]{
\draw[rotate=#1,thick,color=gray] (0,-11) circle (5);
\draw[rotate=#1,thick,color=gray] (0,-11) circle (1);
\draw[rotate=#1,line width=8pt,color=white] (-2,-6.42) -- (2,-6.42);
\draw[rotate=#1,thick,color=gray] (2,-1.15) -- (2,-6.42);
\draw[rotate=#1,thick,color=gray] (-2,-1.15) -- (-2,-6.42);
\draw[rotate=#1] (0,0) -- (0,-8);
\draw[rotate=#1] (0,-11) circle (3);}
\feuille{0}
\feuille{120}
\feuille{-120}
\draw (-4,10) node{$\scriptstyle{leaf}$};
\draw[->] (-5,9) -- (-6.3,7.5);
\draw (11.5,-1) node{$\scriptstyle{internal\ vertex}$};
\draw[<-] (0.5,-0.1) -- (4,-1);
\draw (4.7,-9.2) node{$\Gamma_0$};
\draw[color=gray] (6.5,-16.8) node{$\Sigma(\Gamma_0)$};
\end{tikzpicture}
\end{center}
\caption{the standard Y-graph}\label{figY}
\end{figure}
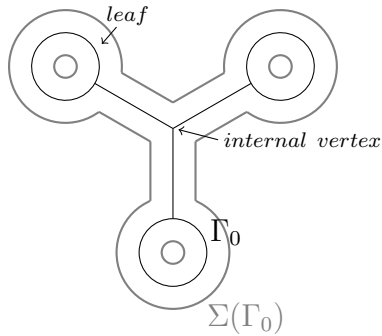
Consider a 3-manifold $M$ and an embedding $h:\Sigma(\Gamma_0)\to M$. The image $\Gamma$ of $\Gamma_0$ is a {\em Y-graph}, and 
$\Sigma(\Gamma)=h(\Sigma(\Gamma_0))$ is the {\em associated surface} of $\Gamma$. The Y-graph $\Gamma$ is equipped with 
the framing induced by $\Sigma(\Gamma)$. The looped edges of a Y-graph are called \emph{leaves}. 
The vertex incident to three different edges is the {\em internal vertex}. 
\begin{figure}[htb] 
\begin{center}
\begin{tikzpicture} [scale=0.15]
\begin{scope}
\newcommand{\feuille}[1]{
\draw[rotate=#1] (0,0) -- (0,-8);
\draw[rotate=#1] (0,-11) circle (3);}
\feuille{0}
\feuille{120}
\feuille{-120}
\draw (3,-4) node{$\Gamma$};
\end{scope}
\draw[very thick,->] (21.5,-3) -- (23.5,-3);
\begin{scope}[xshift=1200]
\newcommand{\bras}[1]{
\draw[rotate=#1] (0,-1.5) circle (2.5);
\draw [rotate=#1,white,line width=8pt] (-0.95,-4) -- (0.95,-4);
\draw[rotate=#1] {(0,-11) circle (3) (1,-3.9) -- (1,-7.6)};
\draw[rotate=#1,white,line width=6pt] (-1,-5) -- (-1,-8.7);
\draw[rotate=#1] {(-1,-3.9) -- (-1,-8.7) (-1,-8.7) arc (-180:0:1)};}
\bras{0}
\draw [white,line width=6pt,rotate=120] (0,-1.5) circle (2.5);
\bras{120}
\draw [rotate=-120,white,line width=6pt] (-1.77,0.27) arc (135:190:2.5);
\draw [rotate=-120,white,line width=6pt] (1.77,0.27) arc (45:90:2.5);
\bras{-120}
\draw [white,line width=6pt] (-1.77,0.27) arc (135:190:2.5);
\draw [white,line width=6pt] (1.77,0.27) arc (45:90:2.5);
\draw (-1.77,0.27) arc (135:190:2.5);
\draw (1.77,0.27) arc (45:90:2.5);
\draw (3.5,-4.5) node{$L$};
\end{scope}
\end{tikzpicture}
\end{center}
\caption{Y-graph and associated surgery link}\label{figborro}
\end{figure}
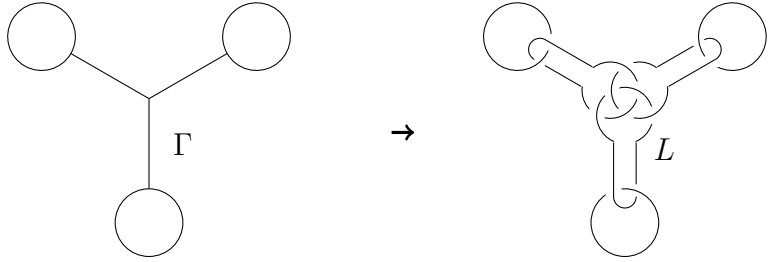

Consider a Y-graph $\Gamma$ in a 3-manifold $M$. Associate with $\Gamma$ the six-component link $L$
represented in Figure \ref{figborro}. 
The \emph{borromean surgery on $\Gamma$} is the surgery along the framed link $L$. 
As proved by Matveev in \cite{Mat}, a borromean surgery can be realised by cutting a genus 3 handlebody 
(a regular neighborhood of the Y-graph) and regluing it another way. A \emph{Y-link} in a 3-manifold is a collection of disjoint Y-graphs.

Consider the rational vector space $\mathcal{F}_0^\mathbb{Z}$ generated by all $\mathbb{Z}$HS's up to orientation-preserving 
homeomorphism. Let $\mathcal{F}_n^\mathbb{Z}$ denote the subspace generated by all the 
$$[M;\Gamma]=\sum_{I\subset\{1,..,n\}}(-1)^{|I|}M(\cup_{i\in I}\Gamma_i),$$
where $M$ is a $\mathbb{Z}$HS, the $\Gamma_i$ are disjoint Y-graphs in $M$, $\Gamma=\cup_{i=1}^n\Gamma_i$, 
and $M(\cup_{i\in I}\Gamma_i)$ is the manifold obtained from $M$ by surgery on the $\Gamma_i$ for $i\in I$.
Here and in all the article, $|I|$ stands for the cardinality of the set $I$.
The associated quotients $\displaystyle \mathcal{G}_n^\mathbb{Z}=\frac{\mathcal{F}_n^\mathbb{Z}}{\mathcal{F}_{n+1}^\mathbb{Z}}$
can be described in terms of Jacobi diagrams.

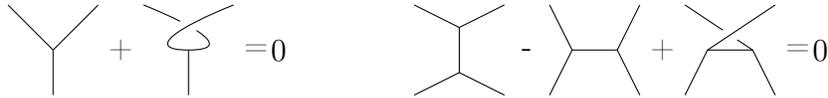
\begin{figure}[htb] 
\begin{center}
\begin{tikzpicture} [scale=0.3]
\draw (0,4) -- (2,2);
\draw (2,2) -- (4,4);
\draw (2,2) -- (2,0);
\draw (5,2) node{+};
\draw (8,2) .. controls +(2,0) and +(2.5,-1) .. (6,4);
\draw[white,line width=6pt] (8,2) .. controls +(-2,0) and +(-2.5,-1) .. (10,4);
\draw (8,2) .. controls +(-2,0) and +(-2.5,-1) .. (10,4);
\draw (8,0) -- (8,2);
\draw (11,2) node{=};
\draw (12,2) node{0};
\draw (18,4) -- (20,3) -- (20,1) -- (18,0);
\draw (20,1) -- (22,0);
\draw (20,3) -- (22,4);
\draw (23,2) node{-};
\draw (24,4) -- (25,2) -- (27,2) -- (28,4);
\draw (24,0) -- (25,2);
\draw (27,2) -- (28,0);
\draw (29,2) node{+};
\draw (30,4) -- (33,2) -- (34,0);
\draw[white,line width=6pt] (31,2) -- (34,4);
\draw (30,0) -- (31,2) -- (34,4);
\draw (31,2) -- (33,2);
\draw (35,2) node{=};
\draw (36,2) node{0};
\end{tikzpicture}
\end{center}
\caption{AS and IHX relations} \label{figasihx}
\end{figure}
A \emph{Jacobi diagram} is a trivalent graph with oriented vertices. An orientation of a vertex of such a diagram is a cyclic order 
of the three half-edges that meet at this vertex. In the pictures, this orientation is induced by the cyclic order  
\raisebox{-1.5ex}{\begin{tikzpicture} [scale=0.2]
\newcommand{\tiers}[1]{
\draw[rotate=#1,color=white,line width=4pt] (0,0) -- (0,-2);
\draw[rotate=#1] (0,0) -- (0,-2);}
\draw (0,0) circle (1);
\draw[<-] (-0.05,1) -- (0.05,1);
\tiers{0}
\tiers{120}
\tiers{-120}
\end{tikzpicture}}. 
The \emph{degree} of a Jacobi diagram is half the number of its vertices. Note that it is an integer.
Let $\mathcal{A}_n$ denote the rational vector space generated by all degree $n$ Jacobi diagrams, quotiented out by the AS 
and IHX relations (Figure \ref{figasihx}). 
The space $\mathcal{A}_0$ is generated by the empty diagram.
Let $\mathcal{A}_n^c$ denote the subspace of $\mathcal{A}_n$ generated by the connected diagrams.

Let $\Gamma$ be a Jacobi diagram of degree $n$. Let $\varphi:\Gamma\hookrightarrow\mathbb{R}^3$ 
be an embedding such that the orthogonal projection on $\mathbb{R}^2\times\{0\}$ of $\varphi(\Gamma)$ is regular, 
and hence induces a framing of $\varphi(\Gamma)$.
\begin{figure}[htb] 
\begin{center}
\begin{tikzpicture} [scale=0.2]
\draw (-36,0) -- (-18,0);
\draw (-36,0) node{$\scriptscriptstyle{\bullet}$};
\draw (-18,0) node{$\scriptscriptstyle{\bullet}$};
\draw[->,line width=1.5pt,>=latex] (-10.5,0) -- (-7.5,0);
\draw (0,0) node{$\scriptscriptstyle{\bullet}$};
\draw (0,0) -- (6,0);
\draw (8,0) circle (2);
\draw[color=white,line width=6pt] (8,0) arc (-180:-90:2);
\draw (10,0) circle (2);
\draw[color=white,line width=6pt] (10,0) arc (0:90:2);
\draw (10,0) arc (0:90:2);
\draw (12,0) -- (18,0);
\draw (18,0) node{$\scriptscriptstyle{\bullet}$};
\end{tikzpicture}
\end{center}
\caption{Replacement of an edge} \label{figedge}
\end{figure}
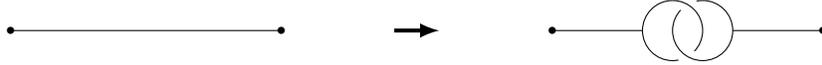 
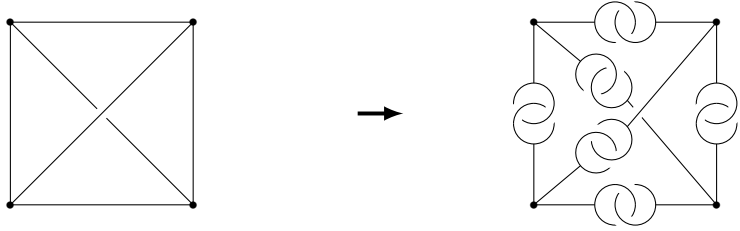
\begin{figure}[thb] 
\begin{center}
\begin{tikzpicture} [scale=0.9]
\begin{scope} [scale=1.35]
\draw (0,2) -- (0,0) -- (2,0) -- (2,2) -- (0,2) -- (2,0);
\draw [color=white,line width=5pt] (0.5,0.5) -- (1.5,1.5);
\draw (0,0) -- (2,2);
\draw (0,0) node{$\scriptscriptstyle{\bullet}$};
\draw (0,2) node{$\scriptscriptstyle{\bullet}$};
\draw (2,0) node{$\scriptscriptstyle{\bullet}$};
\draw (2,2) node{$\scriptscriptstyle{\bullet}$};
\draw[->,line width=1.5pt,>=latex] (3.8,1) -- (4.3,1);
\end{scope}
\begin{scope} [xshift=220,scale=0.15]
\newcommand{\hopf}[2]{
\draw[rotate=#1,yshift=#2] (0,0) -- (6,0);
\draw[rotate=#1,yshift=#2] (8,0) circle (2);
\draw[rotate=#1,yshift=#2,color=white,line width=6pt] (8,0) arc (-180:-90:2);
\draw[rotate=#1,yshift=#2] (10,0) circle (2);
\draw[rotate=#1,yshift=#2,color=white,line width=6pt] (10,0) arc (0:90:2);
\draw[rotate=#1,yshift=#2] (10,0) arc (0:90:2);
\draw[rotate=#1,yshift=#2] (12,0) -- (18,0);}
\hopf{0}{0}
\hopf{90}{0}
\hopf{0}{18cm}
\hopf{90}{-18cm}
\newcommand{\hop}[1]{
\draw[rotate=#1] (0,0) -- (6,0);
\draw[rotate=#1] (8,0) circle (2);
\draw[rotate=#1,color=white,line width=6pt] (8,0) arc (-180:-90:2);
\draw[rotate=#1] (10,0) circle (2);
\draw[rotate=#1,color=white,line width=6pt] (10,0) arc (0:90:2);
\draw[rotate=#1] (10,0) arc (0:90:2);}
\hop{40}
\begin{scope} [yshift=18cm]
\hop{-40}
\end{scope}
\draw (9.2,10.3) -- (18,0);
\draw[color=white,line width=5pt] (9.55,8.35) -- (16.9,16.7);
\draw (9.2,7.7) -- (18,18);
\draw (0,0) node{$\scriptscriptstyle{\bullet}$};
\draw (0,18) node{$\scriptscriptstyle{\bullet}$};
\draw (18,0) node{$\scriptscriptstyle{\bullet}$};
\draw (18,18) node{$\scriptscriptstyle{\bullet}$};
\end{scope}
\end{tikzpicture}
\end{center}
\caption{Jacobi diagram and associated Y-link}
\end{figure}
Now associate a Y-link $\tilde{\Gamma}$ in $S^3$ with $\Gamma$ by replacing all edges of $\varphi(\Gamma)$ as indicated in Figure 
\ref{figedge}. 
\begin{lemma}[GGP, Corollary 4.2, Corollary 4.6, Theorem 4.11] \label{lemmaphi}
 The bracket $[S^3;\tilde{\Gamma}]\in\mathcal{G}_{2n}^\mathbb{Z}$ only depends on 
the class of $\Gamma$ in $\mathcal{A}_n$. Hence it defines: 
$$\begin{array}{cccl}
   \Phi : & \mathcal{A}_n & \to & \mathcal{G}_{2n}^\mathbb{Z} \\
    & \Gamma & \mapsto & [S^3;\Gamma]:=[S^3;\tilde{\Gamma}]
  \end{array}.$$ 
\end{lemma}
Moreover:
\begin{theorem}[Garoufalidis, Goussarov, Polyak \cite{GGP}, Habiro \cite{Hab}, Le \cite{Le}] 
 For $n$ odd, $\mathcal{G}_n^\mathbb{Z}=0$. For $n$ even, the map
$\Phi:\mathcal{A}_{\frac{n}{2}}\to\mathcal{G}_n^\mathbb{Z}$ is an isomorphism. 
\end{theorem}

    \subsection{Statement of the results}

We first define the filtration on the rational vector space $\mathcal{F}_0$ generated by all $\mathbb{Q}$HS's up to 
orientation-preserving homeomorphism.

\begin{definition}
 For $g\in \mathbb{N}$, a \emph{genus $g$ rational (resp. integral) homology handlebody} ($\mathbb{Q}$HH, resp. $\mathbb{Z}$HH) 
is a 3-manifold which is compact, oriented, 
and which has the same homology with rational (resp. integral) coefficients as the standard genus $g$ handlebody.
\end{definition}
Such a $\mathbb{Q}$HH (resp. $\mathbb{Z}$HH) is connected, and its boundary is necessarily homeomorphic to the standard genus $g$ surface.

\begin{definition}
The \emph{Lagrangian} $\mathcal{L}_A$ of a $\mathbb{Q}$HH $A$ is the kernel of the map 
$$i_*: H_1(\partial A;\mathbb{Q})\to H_1(A;\mathbb{Q})$$
induced by the inclusion. Two $\mathbb{Q}$HH's $A$ and $B$ have \emph{LP-identified} boundaries if we have a homeomorphism 
$h:\partial A\to\partial B$ such that $h_*(\mathcal{L}_A)=\mathcal{L}_B$.
\end{definition}
The Lagrangian of a $\mathbb{Q}$HH $A$ is indeed a Lagrangian subspace of $H_1(\partial A;\mathbb{Q})$ 
with respect to the intersection form.

Consider a $\mathbb{Q}$HS $M$, a $\mathbb{Q}$HH $A\subset M$, and a $\mathbb{Q}$HH $B$ whose boundary is LP-identified with $\partial A$.
Set $M(\frac{B}{A})=(M\setminus Int(A))\cup_{\partial A=\partial B}B$. We say that the $\mathbb{Q}$HS 
$M(\frac{B}{A})$ is obtained from $M$ by \emph{Lagrangian preserving surgery}, or \emph{LP-surgery}.
Note that a borromean surgery is a special type of LP-surgery.
 If $(A_i)_{1\leq i\leq n}$ is a family of disjoint $\mathbb{Q}$HH's in $M$, and if, for each $i$, $B_i$ is a $\mathbb{Q}$HH 
whose boundary is LP-identified with $\partial A_i$, we denote by $M((\frac{B_i}{A_i})_{1\leq i\leq n})$ 
the manifold obtained from $M$ by the $n$ LP-surgeries $(\frac{B_i}{A_i})$.

Let $\mathcal{F}_n$ denote the subspace of $\mathcal{F}_0$ generated by the 
$$[M;(\frac{B_i}{A_i})_{1\leq i \leq n}]=\sum_{I\subset \{ 1,...,n\}} (-1)^{|I|} M((\frac{B_i}{A_i})_{i\in I})$$ 
for all $\mathbb{Q}$HS's $M$ and all families of $\mathbb{Q}$HH's $(A_i,B_i)_{1\leq i \leq n}$, where the $A_i$ are embedded in $M$ 
and disjoint, and each $\partial B_i$ is 
LP-identified with the corresponding $\partial A_i$. Since $\mathcal{F}_{n+1}\subset \mathcal{F}_n$, this defines a filtration. 
Set $\mathcal{G}_n=\mathcal{F}_n / \mathcal{F}_{n+1}$ and $\mathcal{G}=\oplus_{n\in\mathbb{N}} \mathcal{G}_n$.
\begin{definition}
 A \emph{finite type invariant of degree at most $n$} of rational homology spheres is a linear map 
$\lambda: \mathcal{F}_0 \to \mathbb{Q}$ such that $\lambda(\mathcal{F}_{n+1})=0$. 
It is said to be \emph{additive} if $\lambda(M\sharp N)=\lambda(M)+\lambda(N)$ for all $\mathbb{Q}$HS's $M$ and $N$.
\end{definition}
Let $\mathcal{I}_n$ (resp. $\mathcal{I}_n^c$) denote the rational vector space of all invariants (resp. additive invariants) 
of degree at most $n$. Set $\mathcal{H}_n = \mathcal{I}_n / \mathcal{I}_{n-1}$ and 
$\mathcal{H}=\oplus_{n\in\mathbb{N}} \mathcal{H}_n$. Note that $\mathcal{I}_n$ 
is canonically isomorphic to $(\mathcal{F}_0 / \mathcal{F}_{n+1})^*:=Hom(\frac{\mathcal{F}_0}{\mathcal{F}_{n+1}},\mathbb{Q})$.
We have an exact sequence:
$$0 \to \mathcal{G}_n \to \frac{\mathcal{F}_0}{\mathcal{F}_{n+1}} \to \frac{\mathcal{F}_0}{\mathcal{F}_n} \to 0.$$
Since the functor $Hom(.,\mathbb{Q})$ is exact, 
the dual sequence $$0 \to \mathcal{I}_{n-1} \to \mathcal{I}_n \to (\mathcal{G}_n)^* \to 0$$ is also exact.
Thus $\mathcal{H}_n \cong (\mathcal{G}_n)^*$.

We will call \emph{augmented diagram of degree $n$} the union of a Jacobi diagram of degree $k\leq\frac{n}{2}$ 
and of $(n-2k)$ weighted vertices, where the weights are prime integers.
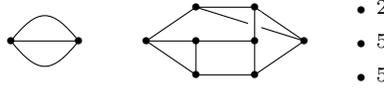
\begin{figure}[htb]
\begin{center}
\begin{tikzpicture} [scale=0.3]
\draw (0,0) -- (3,0) (0,0) .. controls +(1.2,1.5) and +(-1.2,1.5) .. (3,0) (0,0) .. controls +(1.2,-1.5) and +(-1.2,-1.5) .. (3,0);
\draw (6,0) -- (8.2,1.5) -- (10.8,1.5) -- (13,0) -- (10.8,-1.5) -- (8.2,-1.5) -- (6,0) -- (8.2,0) (8.2,1.5) -- (13,0)
(10.8,0) -- (8.2,0) -- (8.2,-1.5);
\draw[color=white,line width=5pt]  (10.8,0.2) -- (10.8,1.3);
\draw  (10.8,1.5) -- (10.8,-1.5);
\draw (16,1.5) node{${\scriptscriptstyle \bullet}{\scriptstyle\ 2}$} 
(16,0) node{${\scriptscriptstyle \bullet}{\scriptstyle\ 5}$} 
(16,-1.5) node{${\scriptscriptstyle \bullet}{\scriptstyle\ 5}$};
\draw (0,0) node{$\scriptscriptstyle{\bullet}$};
\draw (3,0) node{$\scriptscriptstyle{\bullet}$};
\draw (6,0) node{$\scriptscriptstyle{\bullet}$};
\draw (8.2,1.5) node{$\scriptscriptstyle{\bullet}$};
\draw (10.8,1.5) node{$\scriptscriptstyle{\bullet}$};
\draw (13,0) node{$\scriptscriptstyle{\bullet}$};
\draw (10.8,-1.5) node{$\scriptscriptstyle{\bullet}$};
\draw (8.2,-1.5) node{$\scriptscriptstyle{\bullet}$};
\draw (10.8,0) node{$\scriptscriptstyle{\bullet}$};
\draw (8.2,0) node{$\scriptscriptstyle{\bullet}$};
\end{tikzpicture}
\end{center}
\caption{augmented diagram of degree 13} \label{figaug}
\end{figure}
Note that the degree of an augmented diagram is equal to its number of vertices. 
Let $\mathcal{A}_n^{aug}$ denote the rational vector space generated by all augmented diagrams of degree $n$, quotiented out 
by the AS and IHX relations. The main goal of this article is to prove the following theorem:
\begin{theorem} \label{thprinc}
 For $n\in\mathbb{N}$, $\mathcal{A}_n^{aug}\cong\mathcal{G}_n$.
\end{theorem}
This result will follow from Proposition \ref{propdeg1}, Proposition \ref{propdual},
and Proposition \ref{propinduction}. An isomorphism can be described in the following way. Consider an augmented diagram $\Gamma_a$ of 
degree $n$ given by a Jacobi diagram $\Gamma$ of degree $k$, and $(n-2k)$ vertices with weights $(p_i)_{1\leq i\leq n-2k}$. 
Define $\varphi(\Gamma)\subset S^3$ and the associated Y-link $\tilde{\Gamma}$ as before. For each $i$, consider a 
rational homology ball $B_{p_i}$ such that $H_1(B_{p_i};\mathbb{Z})=\frac{\mathbb{Z}}{p_i\mathbb{Z}}$. Then define the image of $\Gamma_a$ as 
$[S^3;\tilde{\Gamma},(\frac{B_{p_i}}{B^3})_{1\leq i\leq n-2k}]\in\mathcal{G}_n$.

Since connected sums are LP-surgeries of genus 0, one can easily see that $\mathcal{G}_0\cong\mathbb{Q}S^3$.
In Section \ref{secdeg1}, we give a description of $\mathcal{G}_1$.
\begin{proposition} \label{propdeg1}
 For any prime integer $p$, fix a $\mathbb{Q}$HS $M_p$ such that $|H_1(M_p)|=p$.
Then $(M_p-S^3)_{p\ prime}$ is a basis for $\mathcal{G}_1$.
\end{proposition}
\paragraph{Remark:}
The $\mathbb{Q}$HS's $M_p$ are not unique in $\mathcal{F}_0$, but we will see in Subsection \ref{subsecgene} that they are 
unique modulo $\mathcal{F}_2$.

We will show in Subsection \ref{subsecgene} that the family $(M_p-S^3)_{p\ prime}$ generates $\mathcal{G}_1$.
To see that it is a basis, we will prove the following proposition in Subsection \ref{subsecinv}.

For a prime integer $p$, let $v_p$ denote the $p$-adic valuation, defined on $\mathbb{N}\setminus \{0\}$ 
by $v_p(p^kn)=k$ if $n$ is prime to $p$.
\begin{proposition} \label{propinv}
 For any prime $p$, define a linear map $\nu_p$ on $\mathcal{F}_0$ by setting $\nu_p(M)=v_p(|H_1(M)|)$ when $M$ is a $\mathbb{Q}$HS. 
Then $\nu_p$ is a degree 1 invariant of $\mathbb{Q}$HS's.
\end{proposition}
Since $\nu_p(M_p)=1$, $\nu_p(M_q)=0$ for any prime $q\neq p$, and $\nu_p(S^3)=0$ for any prime $p$,
this result shows that the family $(M_p-S^3)_{p\ prime}$ is free.

\begin{corollary} \label{cordeg1}
 $\displaystyle \frac{\mathcal{I}_1}{\mathcal{I}_0}=\mathcal{I}_1^c= \prod_{p\ prime} \mathbb{Q}\,\nu_p$.
\end{corollary}

In Section \ref{secadd}, we prove:
\begin{proposition} \label{propdual}
 For $n>1$, $\displaystyle \frac{\mathcal{I}_n^c}{\mathcal{I}_{n-1}^c}\cong(\mathcal{A}_\frac{n}{2}^c)^*$ if $n$ is even,
and $\displaystyle \frac{\mathcal{I}_n^c}{\mathcal{I}_{n-1}^c}\cong 0$ if $n$ is odd.
\end{proposition}
In this proof, we will use the description of finite type invariants of degree 1 of framed rational homology tori 
given in Subsection \ref{subsectori}.

In Section \ref{secfil}, we use the structures of graded algebras on $\mathcal{G}$ and $\mathcal{H}$ in order to show 
that any finite type invariant $\lambda$ such that $\lambda(S^3)=0$ can be written as a sum of products of additive invariants.
More precisely, let $\mathcal{I}_n^\pi$ denote the subspace of $\mathcal{I}_n$ generated by all the products 
$\prod_{1\leq i\leq k} \lambda_i$, where $k>1$, the $\lambda_i$ are additive invariants of degree $k_i<n$, 
and $\sum_{1\leq i\leq k} k_i\leq n$. 
Our version of the Milnor-Moore theorem about the structure of Hopf algebras implies that:
\begin{proposition} \label{propadd} 
For all $n>0$, $\mathcal{I}_n=\mathcal{I}_0\oplus\mathcal{I}_n^c\oplus\mathcal{I}_n^\pi$.
\end{proposition}
We will obtain this result as a consequence of Proposition \ref{propinduction}.

In order to describe the spaces of additive invariants, we shall prove that LP-surgeries can be reduced to more specific moves.
\begin{definition}
 Consider a positive integer $d$. We call \emph{$d$-torus} a rational homology torus such that: 
\begin{itemize}
 \item $H_1(\partial T_d;\mathbb{Z})=\mathbb{Z} \alpha \oplus \mathbb{Z}\beta$, with $<\alpha,\beta>=1$,
 \item $d\alpha=0$ in $H_1(T_d;\mathbb{Z})$,
 \item $\beta=d\gamma$ in $H_1(T_d;\mathbb{Z})$, where $\gamma$ is a curve in $T_d$,
 \item $H_1(T_d;\mathbb{Z})=\mathbb{Z}_d \alpha \oplus \mathbb{Z} \gamma$.
\end{itemize}
\end{definition}

\begin{definition}
 An \emph{elementary surgery} is an LP-surgery among the following ones:
\begin{enumerate}
 \item connected sum (genus 0),
 \item LP-replacement of a standard torus by a $d$-torus (genus 1),
 \item borromean surgery (genus 3).
\end{enumerate}
\end{definition}

In Section \ref{secelsur}, we prove:
\begin{theorem} \label{thelsur}
 If $A$ and $B$ are two $\mathbb{Q}$HH's with LP-identified boundaries, then $B$ can be obtained from $A$ 
by a finite sequence of elementary surgeries and their inverses in the interior of the $\mathbb{Q}$HH's.
\end{theorem}
This proposition generalises a result of Auclair and Lescop which says that any two $\mathbb{Z}$HH's 
with LP-identified boundaries can be obtained from one another by a finite sequence of borromean surgeries in the interior 
of the $\mathbb{Z}$HH's (\cite[Lemma 4.11]{AL}).

In Section \ref{secborro}, we recall some facts about borromean surgeries proved in \cite{GGP}, and we give consequences 
of these facts that are useful in the sequel.

\paragraph{Acknowledgements} I wish to thank the referee for his careful reading. My thanks also go to my advisor, Christine Lescop, 
for her helpful advice and rigorous supervision. 

    \section{Elementary surgeries} \label{secelsur}

      \subsection{Homological properties of $\mathbb{Q}$HH's}

\begin{definition}
 Consider the genus $g$ compact surface $\Sigma_g$. A basis $(\alpha_i,\beta_i)_{1\leq i\leq g}$ of $H_1(\Sigma_g;\mathbb{Z})$
is called \emph{symplectic} if the matrix of the intersection form in $(\alpha_1,..\alpha_g,\beta_1,..,\beta_g)$ 
is $\begin{pmatrix} 0 & I_g \\ -I_g & 0 \end{pmatrix}$.
\end{definition}

\paragraph{Notation} We denote by $Tors(H)$ the torsion submodule of a module $H$.

\begin{lemma} \label{lemmaprel}
If $A$ is a genus $g$ $\mathbb{Q}$HH, then:
\begin{itemize}
 \item $H_1(A;\mathbb{Z})\cong\mathbb{Z}^g\oplus Tors(H_1(A;\mathbb{Z}))$,
 \item $H_2(A;\mathbb{Z})=0$,
 \item $H_2(A,\partial A;\mathbb{Z})\cong (\frac{H_1(A;\mathbb{Z})}{Tors(H_1(A;\mathbb{Z}))})^*\cong\mathbb{Z}^g$.
\end{itemize} 
\end{lemma}
\begin{proof}
The first point is given by $H_1(A;\mathbb{Z})\otimes\mathbb{Q}\cong H_1(A;\mathbb{Q})\cong\mathbb{Q}^g$.

By the Poincar\'e duality, we have $H_2(A;\mathbb{Z})\cong H^1(A,\partial A;\mathbb{Z})$. The universal coefficient theorem gives
$H^1(A,\partial A;\mathbb{Z})\cong Hom(H_1(A,\partial A;\mathbb{Z}),\mathbb{Z})$. Hence
$H_2(A;\mathbb{Z})$ is torsion free. Since $H_2(A;\mathbb{Q})=0$, we get the second point.

The last point also follows from the Poincar\'e duality and the universal coefficient theorem:
$$H_2(A,\partial A;\mathbb{Z})\cong H^1(A;\mathbb{Z})\cong Hom(H_1(A;\mathbb{Z}),\mathbb{Z})\cong\mathbb{Z}^g.$$ 
\end{proof}

\begin{lemma} \label{lemmaduality}
Consider a genus $g$ $\mathbb{Q}$HH $A$. Consider the map $i_*: H_1(\partial A;\mathbb{Z}) \to H_1(A;\mathbb{Z})$ induced 
by the inclusion. Set: $$\mathcal{L}_A^\mathbb{Z}=Ker(i_*),\quad\mathcal{L}_A^T=(i_*)^{-1}(Tors(H_1(A;\mathbb{Z}))).$$
Then there is a symplectic basis $(\alpha_i,\beta_i)_{1\leq i\leq g}$ of $H_1(\partial A;\mathbb{Z})$, 
a family $(\gamma_i)_{1\leq i\leq g}$ of curves in $A$, and positive integers $d_i$, ${1\leq i\leq g}$, such that:
$$\mathcal{L}_A^\mathbb{Z}=\bigoplus_{1\leq i\leq g}\mathbb{Z}(d_i\alpha_i), \quad 
\mathcal{L}_A^T=\bigoplus_{1\leq i\leq g} \mathbb{Z} \alpha_i, \quad 
\frac{H_1(A;\mathbb{Z})}{Tors(H_1(A;\mathbb{Z}))}=\bigoplus_{1\leq i\leq g} \mathbb{Z}\gamma_i,$$ 
and $\beta_i=d_i\gamma_i$ in $\frac{H_1(A;\mathbb{Z})}{Tors(H_1(A;\mathbb{Z}))}$ for $1\leq i\leq g$.

In particular, $\displaystyle \frac{\mathcal{L}_A^T}{\mathcal{L}_A^\mathbb{Z}}$ and 
$\displaystyle \frac{H_1(A;\mathbb{Z})}{Tors(H_1(A;\mathbb{Z}))\oplus (\oplus_{1\leq i\leq g} \mathbb{Z} \beta_i)}$
are isomorphic to $\displaystyle \prod_{1\leq i\leq g} \frac{\mathbb{Z}}{d_i\mathbb{Z}}$.
\end{lemma}
\begin{proof}
The exact sequence over $\mathbb{Z}$ associated with $(A,\partial A)$ yields the following exact sequence:
$$0 \to H_2(A,\partial A) \to H_1(\partial A) \fl{i_*} H_1(A),$$ 
thus $\mathcal{L}_A^\mathbb{Z}$ is a free submodule of rank $g$ of $H_1(\partial A;\mathbb{Z})$. 
Hence there is a basis $(\alpha_i,\beta_i)_{1\leq i\leq g}$ of $H_1(\partial A;\mathbb{Z})$, 
and positive integers $d_i$, ${1\leq i\leq g}$, such that $(d_i\alpha_i)_{1\leq i\leq g}$ is a basis of $\mathcal{L}_A^\mathbb{Z}$.
It follows that $\mathcal{L}_A^T=\bigoplus_{1\leq i\leq g} \mathbb{Z} \alpha_i$.
Since the intersection form is trivial on $\mathcal{L}_A^T$, we can choose the $\beta_i$ 
in such a way that the basis $(\alpha_i,\beta_i)_{1\leq i\leq g}$ is symplectic.

The boundary map $H_2(A,\partial A) \to H_1(\partial A)$ in the above exact sequence induces an isomorphism 
$H_2(A,\partial A;\mathbb{Z})\cong\mathcal{L}_A^\mathbb{Z}$.
Thus we can choose a basis $(S_i)_{1\leq i\leq g}$ of $H_2(A,\partial A;\mathbb{Z})$ such that 
$\partial S_i=d_i\alpha_i$ for $1\leq i\leq g$. 
Let $(\gamma_i)_{1\leq i\leq g}$ denote the basis of $\frac{H_1(A;\mathbb{Z})}{Tors(H_1(A;\mathbb{Z}))}$ 
Poincar\'e dual to $(S_i)_{1\leq i\leq g}$.

For $1\leq i,j\leq g$, $<S_j,\beta_i>_A=<d_j\alpha_j,\beta_i>_{\partial A}=\delta_{ij} d_i$, where $\delta_{ij}$ is the Kronecker 
delta, equal to 1 if $i=j$, and 0 otherwise. Thus $\beta_i=d_i\gamma_i$
in $\frac{H_1(A;\mathbb{Z})}{Tors(H_1(A;\mathbb{Z}))}$.
\end{proof}

\begin{corollary} \label{corz}
 Let $A$ be a $\mathbb{Q}$HH. If the map $H_1(\partial A;\mathbb{Z})\to H_1(A;\mathbb{Z})$, induced by the inclusion 
$\partial A \hookrightarrow A$, is surjective, then $A$ is a $\mathbb{Z}$HH.
\end{corollary}

      \subsection{$d$-tori}

\begin{lemma} \label{lemmaTd}
 For any positive integer $d$, there exists a $d$-torus $T_d$.
\end{lemma}
\begin{proof}
 Consider the standard genus 2 handlebody $A$ represented in Figure \ref{figA}.
\begin{figure}[htb] 
\begin{center}
\begin{tikzpicture} [scale=0.5]
\draw (0,0) ..controls +(0,1) and +(-2,1) .. (4,2);
\draw (4,2) ..controls +(2,-1) and +(-2,-1) .. (8,2);
\draw (8,2) ..controls +(2,1) and +(0,1) .. (12,0);
\draw (0,0) ..controls +(0,-1) and +(-2,-1) .. (4,-2);
\draw (4,-2) ..controls +(2,1) and +(-2,1) .. (8,-2);
\draw (8,-2) ..controls +(2,-1) and +(0,-1) .. (12,0);

\draw (2,0) ..controls +(0.5,-0.25) and +(-0.5,-0.25) .. (4,0);
\draw (2.3,-0.1) ..controls +(0.6,0.2) and +(-0.6,0.2) .. (3.7,-0.1);
\draw (3,-0.2) ..controls +(0.2,-0.5) and +(0.2,0.5) .. (3,-2.3);
\draw[dashed] (3,-0.2) ..controls +(-0.2,-0.5) and +(-0.2,0.5) .. (3,-2.3);
\draw (3,0)ellipse(1.6 and 0.8);
\draw [->] (3.15,-1.2)--(3.15,-1.3);
\draw (3.65,-1.3) node {$\scriptstyle{a_1}$};
\draw [->] (4.6,-0.05)--(4.6,0.05);
\draw (5.1,0) node {$\scriptstyle{b_1}$};

\draw (8,0) ..controls +(0.5,-0.25) and +(-0.5,-0.25) .. (10,0);
\draw (8.3,-0.1) ..controls +(0.6,0.2) and +(-0.6,0.2) .. (9.7,-0.1);
\draw (9,-0.2) ..controls +(0.2,-0.5) and +(0.2,0.5) .. (9,-2.3);
\draw[dashed] (9,-0.2) ..controls +(-0.2,-0.5) and +(-0.2,0.5) .. (9,-2.3);
\draw (9,0)ellipse(1.6 and 0.8);
\draw [->] (9.15,-1.2)--(9.15,-1.3);
\draw (9.65,-1.3) node {$\scriptstyle{a_2}$};
\draw [->] (10.6,-0.05)--(10.6,0.05);
\draw (11.1,0) node {$\scriptstyle{b_2}$};
\end{tikzpicture}
\caption{The handlebody $A$}\label{figA}
\end{center}
\end{figure}
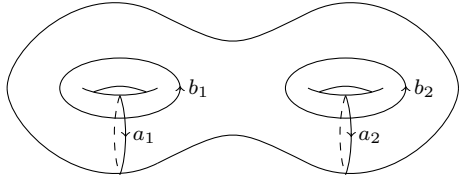
 Consider a curve $c$ on $\partial A$ such that $c=a_1+db_2$ in $H_1(\partial A;\mathbb{Z})$. According to Meyerson in \cite{Mey}, 
since $c$ is primitive, it can be chosen simple and closed. The torus $T_d$ will be obtained from $A$ by adding a 2-handle to $A$ 
along $c$ as follows. Define $T_d=A\cup_h (D^2\times [-1,1])$, where $h: \partial D^2\times[-1,1]\to \partial A$ is an embedding 
such that $h(\partial D^2\times \{0\})=c$.
We have $H_1(T_d;\mathbb{Z})=<b_1,b_2\,|\,db_2=0>$.

Moreover, we can define curves $\alpha$, $\beta$, $\gamma$, on $\partial A$, 
\begin{figure}[htb] 
\begin{center}
\begin{tikzpicture} [scale=0.5]
\draw (0,0) ..controls +(0,1) and +(-2,1) .. (4,2);
\draw (4,2) ..controls +(2,-1) and +(-2,-1) .. (8,2);
\draw (8,2) ..controls +(2,1) and +(0,1) .. (12,0);
\draw (0,0) ..controls +(0,-1) and +(-2,-1) .. (4,-2);
\draw (4,-2) ..controls +(2,1) and +(-2,1) .. (8,-2);
\draw (8,-2) ..controls +(2,-1) and +(0,-1) .. (12,0);

\draw (2,0) ..controls +(0.5,-0.25) and +(-0.5,-0.25) .. (4,0);
\draw (2.3,-0.1) ..controls +(0.6,0.2) and +(-0.6,0.2) .. (3.7,-0.1);
\draw (3,-0.2) ..controls +(0.2,-0.5) and +(0.2,0.5) .. (3,-2.3);
\draw[dashed] (3,-0.2) ..controls +(-0.2,-0.5) and +(-0.2,0.5) .. (3,-2.3);
\draw (3,0)ellipse(1.6 and 0.8);
\draw [->] (3.15,-1.2)--(3.15,-1.3);
\draw (3.6,-1.3) node {$\scriptstyle{\gamma}$};
\draw [->] (4.6,-0.05)--(4.6,0.05);
\draw (5.1,0) node {$\scriptstyle{c}$};

\draw (8,0) ..controls +(0.5,-0.25) and +(-0.5,-0.25) .. (10,0);
\draw (8.3,-0.1) ..controls +(0.6,0.2) and +(-0.6,0.2) .. (9.7,-0.1);
\draw (9,-0.2) ..controls +(0.2,-0.5) and +(0.2,0.5) .. (9,-2.3);
\draw[dashed] (9,-0.2) ..controls +(-0.2,-0.5) and +(-0.2,0.5) .. (9,-2.3);
\draw (9,0)ellipse(1.6 and 0.8);
\draw [->] (9.15,-1.2)--(9.15,-1.3);
\draw (9.6,-1.3) node {$\scriptstyle{\alpha}$};
\draw [->] (10.6,-0.05)--(10.6,0.05);
\draw (11.1,0) node {$\scriptstyle{\beta}$};
\end{tikzpicture}
\caption{The surface $\partial A$}\label{figbord}
\end{center}
\end{figure}
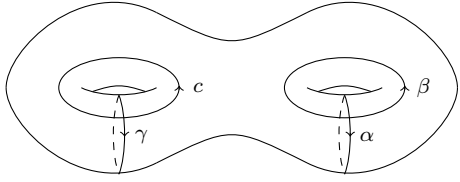
with $\alpha=b_2$, $\beta=-a_2-db_1$ and $\gamma=-b_1$ in $H_1(\partial A;\mathbb{Z})$ such that the boundary of $A$ 
is homeomorphic to the surface represented in Figure \ref{figbord}.
We get $H_1(T_d;\mathbb{Z})=<\gamma,\alpha\,|\,d\alpha=0>$, 
and $H_1(\partial T_d;\mathbb{Z})=\mathbb{Z}\alpha\oplus\mathbb{Z}\beta$.\end{proof}

Given a curve $\gamma$ in a 3-manifold $M$, we will call {\em exterior of $\gamma$ in $M$} the complement of the open tubular 
neighborhood of $\gamma$ in $M$.

\begin{lemma} \label{lemmameridian}
  Let $d$ be a positive integer. Let $T_d$ be a $d$-torus. Let  $\gamma$ be 
a curve in $T_d$ whose homology class generates $\frac{H_1(T_d;\mathbb{Z})}{Tors(H_1(T_d;\mathbb{Z}))}$. 
Let $m(\gamma)$ and $\ell(\gamma)$ be respectively a meridian and a parallel of $\gamma$. 
For any integer $k$, there is a symplectic basis $(\alpha,\beta)$ of $H_1(\partial T_d;\mathbb{Z})$ such that $d\alpha=0$ 
in $H_1(T_d;\mathbb{Z})$, and such that the curve $\beta-d\ell(\gamma)+km(\gamma)$ bounds a surface in the exterior of $\gamma$ in $T_d$.
\end{lemma}
\begin{proof}
 Let $X$ be the exterior of $\gamma$ in $T_d$. Consider a symplectic basis $(\alpha,\beta_0)$ of $H_1(\partial T_d;\mathbb{Z})$ 
such that $d\alpha=0$ and $\beta_0=d\gamma$ in $H_1(T_d;\mathbb{Z})$. 
There is an integer $k_0$ such that $\beta_0-d\ell(\gamma)+k_0m(\gamma)$ bounds a surface in $X$, {\em i.e.} is trivial in $H_1(X;\mathbb{Z})$. 
Since $d\alpha$ bounds a surface in $T_d$ that $\gamma$ meets once, $d\alpha-m(\gamma)$ is trivial in $H_1(X;\mathbb{Z})$. 
Let $k$ be any integer, and set $\beta=\beta_0+(k_0-k)d\alpha$. 
The curve $\beta-d\ell(\gamma)+km(\gamma)=\lbp\beta_0-d\ell(\gamma)+k_0m(\gamma)\rbp+(k_0-k)\lbp d\alpha-m(\gamma)\rbp$ 
is trivial in $H_1(X;\mathbb{Z})$. 
\end{proof}

      \subsection{Relating $\mathbb{Q}$HH's by elementary surgeries}

In this subsection, we prove Theorem \ref{thelsur}.

\begin{definition}
Consider a $\mathbb{Q}HH$ $A$.
Consider a simple closed curve $\gamma\subset A$.
Consider a disk $D\subset\partial A$.
Consider two distinct points $y$ and $z$ in $Int(D)$, and a path $s$ from $z$ to $y$ in $Int(D)$.
Consider a cylinder $C=h(D^2\times [0,1])\subset A$, where $h$ is an embedding such that:
\begin{itemize}
 \item $h(D^2\times\{0\})$ (resp. $h(D^2\times\{1\})$) is a disk $D_y$ (resp. $D_z$) in $Int(D)$,
 \item $h(0,0)=y$ and $h(0,1)=z$,
 \item $C\cap\partial A=D_y\cup D_z$,
 \item $h(\{0\}\times[0,1])\cup s$ is homologous to $\gamma$ in $A$.
\end{itemize}
We will call \emph{tunnel around $\gamma$} such a cylinder $C$.
\end{definition}

\begin{lemma} \label{lemma1}
 Let $A$ be a $\mathbb{Q}$HH of genus $g$. Let $\gamma$ be a simple closed curve in $A$. Let $C$ be a tunnel around $\gamma$. 
Set $B=\overline{A\setminus C}$. Then $B$ is a $\mathbb{Q}$HH of genus $g+1$.
\end{lemma}
\begin{proof}
 Consider the pair $(A,B)$. By excision, for $i\in\mathbb{N}$, $H_i(A,B;\mathbb{Q})\cong H_i(C,C\cap B;\mathbb{Q})$. 
Since $(C,C\cap B)\cong(D^2\times[0,1],(\partial D^2)\times[0,1])$, it follows that 
$H_i(A,B;\mathbb{Q})=0$ if $i\neq 2$, and $H_2(A,B;\mathbb{Q})\cong \mathbb{Q}$. The exact sequence over $\mathbb{Q}$ associated 
with the pair $(A,B)$ yields the following exact sequence:
$$0\to H_2(B)\to 0\to H_2(A,B)\cong\mathbb{Q}\to H_1(B)\to H_1(A)\cong \mathbb{Q}^g\to 0.$$
Hence $H_2(B;\mathbb{Q})=0$ and $H_1(B;\mathbb{Q})\cong \mathbb{Q}^{g+1}$.
\end{proof}

\begin{lemma} \label{lemma2}
 Let $A$ be a $\mathbb{Q}$HH of genus $g$. The quotient $\displaystyle \frac{H_1(A;\mathbb{Z})}{H_1(\partial A;\mathbb{Z})}$ is a torsion 
module. Set $\displaystyle \frac{H_1(A;\mathbb{Z})}{H_1(\partial A;\mathbb{Z})}=\oplus_{i=1}^n \frac{\mathbb{Z}}{d_i\mathbb{Z}} \mu_i$. 
Let $C_i$, $1\leq i\leq n$, be pairwise disjoint tunnels around the $\mu_i$. 
Then $B=\overline{A\setminus (\cup_{1\leq i\leq n} C_i)}$ is a $\mathbb{Z}$HH of genus $g+n$.
\end{lemma}
\begin{proof}
The fact that $\displaystyle \frac{H_1(A;\mathbb{Z})}{H_1(\partial A;\mathbb{Z})}$ is a torsion module follows from Lemma \ref{lemmaduality}.

By Lemma \ref{lemma1}, $B$ is a $\mathbb{Q}$HH of genus $g+n$. Hence, by Corollary \ref{corz}, it suffices to show that 
the map $H_1(\partial B;\mathbb{Z}) \to H_1(B;\mathbb{Z})$ induced by the inclusion is surjective, or, equivalently, 
that $H_1(B,\partial B;\mathbb{Z})$ is trivial. By excision, $H_1(B,\partial B;\mathbb{Z})$ is isomorphic to 
$H_1(A,\partial A \cup (\cup_{1\leq i\leq n} C_i);\mathbb{Z})$, which is trivial by definition of the $C_i$'s.
\end{proof}

For a 3-manifold $A$, let $lk_A : Tors(H_1(A;\mathbb{Z}))\times Tors(H_1(A;\mathbb{Z})) \to \mathbb{Q}/\mathbb{Z}$ denote the {\em linking form} 
on $A$, defined in the following way. Consider disjoint representatives $\alpha,\beta$ of two homology classes
in $Tors(H_1(A;\mathbb{Z}))$. Consider a surface $S\subset A$, transverse to $\beta$, such that $\partial S=k\alpha$ 
for some positive integer $k$. Then $lk_A(\alpha,\beta)=\frac{1}{k}<S,\beta>$, where $<.,.>$ is the algebraic intersection 
number in $A$. For a $\mathbb{Q}$HS $M$, the linking form $lk_M$ is defined on $H_1(M;\mathbb{Z})\times H_1(M;\mathbb{Z})$, and it is known 
to be bilinear, symmetric, and non degenerate. 

\begin{lemma} \label{lemma3}
Let $A$ be a $\mathbb{Q}$HH of genus $g$. Assume $\frac{\mathcal{L}_A^T}{\mathcal{L}_A^\mathbb{Z}}=0$. 
Then there exists a $\mathbb{Q}$HS $M$ such that $(H_1(M;\mathbb{Z}),lk_M)$ is isomorphic to $(Tors(H_1(A;\mathbb{Z})),lk_A)$.
\end{lemma}
\begin{proof}
 By Lemma \ref{lemmaduality}, there is a symplectic basis $(\alpha_i,\beta_i)_{1\leq i\leq g}$ of $H_1(\partial A;\mathbb{Z})$ 
such that the $\alpha_i$ are null-homologous in $A$, and $H_1(A;\mathbb{Z})=Tors(H_1(A))\oplus(\oplus_{1\leq i\leq g} \mathbb{Z}\beta_i)$.
Consider a standard handlebody $H_g$, and a symplectic basis $(a_i,b_i)_{1\leq i\leq g}$ of $H_1(\partial H_g;\mathbb{Z})$,
where each $a_i$ bounds a disk in $H_g$. Construct a $\mathbb{Q}$HS $M$ by gluing $A$ and $H_g$ 
along their boundaries, in such a way that, for $1\leq i \leq g$, $\alpha_i$ is identified with $b_i$, and $\beta_i$ is 
identified with $a_i$. We have $H_1(M;\mathbb{Z})\cong Tors(H_1(A;\mathbb{Z}))$. Moreover, the linkings of the curves in $A$ are 
preserved, thus the linking forms on $H_1(M)$ and $Tors(H_1(A))$ are isomorphic.
\end{proof}

\begin{lemma} \label{lemma4}
Let $A$ and $A'$ be $\mathbb{Q}$HH's of genus $g$ with LP-identified boundaries. Assume $\frac{\mathcal{L}_A^T}{\mathcal{L}_A^\mathbb{Z}}=0$ and 
$\frac{\mathcal{L}_{A'}^T}{\mathcal{L}_{A'}^\mathbb{Z}}=0$. 
If $(Tors(H_1(A)),lk_A)$ is isomorphic to $(Tors(H_1(A')),lk_{A'})$, then $A$ and $A'$ can be obtained from one another 
by a finite sequence of borromean surgeries.
\end{lemma}
\begin{proof}
Consider a basis $(\mu_i)_{1\leq i\leq n}$ of $Tors(H_1(A))$, and its image $(\mu_i')_{1\leq i\leq n}$ under an isomorphism 
$(Tors(H_1(A)),lk_A)\cong (Tors(H_1(A')),lk_{A'})$. Fix framed representatives of the $\mu_i$ and $\mu_i'$ such that 
$lk(\mu_i,\mu_j)=lk(\mu_i',\mu_j')\in\mathbb{Q}$ for $1\leq i,j\leq n$. Consider pairwise disjoint tunnels $C_i$ (resp. $C_i'$) around 
the $\mu_i$ (resp. $\mu_i'$). Set $B=\overline{A\setminus (\cup_{1\leq i\leq n} C_i)}$ and 
$B'=\overline{A'\setminus (\cup_{1\leq i\leq n} C_i')}$. Extend the identification $\partial A\cong \partial A'$ to an identification 
$\partial B\cong\partial B'$ so that 
the longitude of each $\mu_i$ is identified with the longitude of the corresponding $\mu_i'$. By Lemma \ref{lemma2}, 
$B$ and $B'$ are $\mathbb{Z}$HH's of genus $g+n$. The equality between the linking numbers ensures that the identification 
of their boundaries preserves the Lagrangian. Thus, 
by \cite[Lemma 4.11]{AL}, $B$ can be obtained from $B'$ by a finite sequence of borromean surgeries. Gluing back the cylinders, 
we get that $A$ can be obtained from $A'$ by a finite sequence of borromean surgeries.
\end{proof}

\begin{corollary} \label{corcaspart}
 Consider a $\mathbb{Q}$HH $A$ such that $\mathcal{L}_A^T/\mathcal{L}_A^\mathbb{Z}=0$. 
Let $H_g$ be a standard handlebody such that $\partial H_g$ and $\partial A$ are LP-identified.
Then there exists a $\mathbb{Q}$HS $M$ such that $A$ is obtained from $H_g\sharp M$ by a finite sequence of borromean surgeries.
\end{corollary}

\begin{lemma} \label{lemmacaspart2}
 Let $A$ be a genus $g$ $\mathbb{Q}$HH. Let $H_g$ be a standard handlebody such that $\partial H_g$ and $\partial A$ are LP-identified.
Assume there are a symplectic basis $(\alpha_i,\beta_i)_{1\leq i\leq g}$ 
of $H_1(\partial A;\mathbb{Z})$, a curve $\gamma$ in $A$, and a positive integer $d$ such that 
$H_1(A;\mathbb{Z})=\frac{\mathbb{Z}}{d\mathbb{Z}}\alpha_1\oplus\mathbb{Z}\gamma\oplus(\bigoplus_{2\leq i\leq g}\mathbb{Z}\beta_i)$ 
and $\beta_1=d\gamma$. Then there are a solid torus $T_0$ embedded in $H_g$, a $d$-torus $T_d$, and an LP-identification 
$\partial T_d\cong\partial T_0$, such that $A$ is obtained from $H_g(\frac{T_d}{T_0})$ by a finite sequence of borromean surgeries.
\end{lemma}
\begin{proof}
 Consider a tunnel $C$ around $\gamma$ in $A$. Set $B=\overline{A\setminus C}$. By Lemma \ref{lemma2}, $B$ is a $\mathbb{Z}$HH 
of genus $g+1$. There is a surface $S\subset B$ such that 
$\partial S\subset \partial B$ is homologous to $\beta_1 -d\ell+km$ in $\partial B$, where $m$ is a meridian of $\gamma$, $\ell$ is a longitude
 of $\gamma$, and $k$ is an integer. Consider simple closed curves $\sigma_1$ and $\sigma_2$ in $\partial B$
such that $\sigma_1=m-d\alpha_1$ and $\sigma_2=\beta_1 -d\ell+km$ in $H_1(\partial B)$.
Then $(\sigma_1,\sigma_2,\alpha_2,\dots,\alpha_g)$ is a basis of $\mathcal{L}_B^\mathbb{Z}$.

Consider the symplectic basis $(a_i,b_i)_{1\leq i\leq g}$ 
of $H_1(\partial H_g;\mathbb{Z})$ image of $(\alpha_i,\beta_i)_{1\leq i\leq g}$ by the LP-identification $\partial A\cong\partial H_g$. 
Consider a simple closed curve representing $b_1$
in $Int(H_g)$, and a tubular neighborhood $T_0$ of this curve. Consider a $d$-torus $T_d$, a symplectic basis $(\alpha',\beta')$ 
of $H_1(\partial T_d;\mathbb{Z})$, and a curve $\gamma'$ in $T_d$, such that 
$H_1(T_d;\mathbb{Z})=\frac{\mathbb{Z}}{d\mathbb{Z}}\alpha'\oplus\mathbb{Z}\gamma'$ and $\beta'=d\gamma'$. 
By Lemma \ref{lemmameridian}, $\beta'$ can be chosen so that $\beta'-d\ell(\gamma')+km(\gamma')$ bounds a surface 
in the exterior of $\gamma'$ in $T_d$ (where $k$ is the integer that appears when tunneling $A$).
Choose an LP-identification $\partial T_d\cong\partial T_0$ that identifies $\beta'$ with a curve 
on $\partial T_0$ homologous to $b_1$ in $H_g\setminus Int(T_0)$. Set $A'=H_g(\frac{T_d}{T_0})$.

Consider a tunnel $C'$ around $\gamma'$ in $A'$. Set $B'=\overline{A'\setminus C'}$. By Lemma \ref{lemma2}, $B'$ is a $\mathbb{Z}$HH 
of genus $g+1$. Like in $B$, there is a surface $S'$ in $B'$ 
bounded by $b_1 -d\ell(\gamma')+km(\gamma')$, and we can define a basis of $\mathcal{L}_{B'}^\mathbb{Z}$ similarly. 
Hence the LP-identification $\partial A\cong\partial H_g\cong \partial A'$ 
extends to an LP-identification $\partial B\cong\partial B'$. By \cite[Lemma 4.11]{AL}, $B$ can be obtained from $B'$ 
by a finite sequence of borromean surgeries. Gluing back the cylinders, we get that $A$ can be obtained from $A'=H_g(\frac{T_d}{T_0})$ 
by a finite sequence of borromean surgeries.
\end{proof}

\proofof{Theorem \ref{thelsur}}
It suffices to prove the result when $B$ is a standard handlebody.
We will proceed by induction on $|\mathcal{L}_A^T/\mathcal{L}_A^\mathbb{Z}|$. 
The case $|\mathcal{L}_A^T/\mathcal{L}_A^\mathbb{Z}|=1$ is given by Corollary \ref{corcaspart}.

Consider a $\mathbb{Q}$HH $A$ of genus $g$ with $|\mathcal{L}_A^T/\mathcal{L}_A^\mathbb{Z}|>1$, and a standard genus $g$ 
handlebody $H_g$ whose boundary is LP-identified with $\partial A$. 
By Lemma \ref{lemmaduality}, there is a symplectic basis $(\alpha_i,\beta_i)_{1\leq i\leq g}$ of $H_1(\partial A;\mathbb{Z})$, 
positive integers $d_i$, and a basis $(\gamma_i)_{1\leq i\leq g}$ of $H_1(A;\mathbb{Z})/Tors(H_1(A;\mathbb{Z}))$, 
such that, in $H_1(A;\mathbb{Z})$, $d_i\alpha_i=0$ and $\beta_i=d_i\gamma_i + t_i$, with $t_i\in Tors(H_1(A;\mathbb{Z}))$. 
Note that $|\mathcal{L}_A^T/\mathcal{L}_A^\mathbb{Z}|=\prod_{1\leq i\leq g} d_i$. Assume $d_1>1$.

Consider a tubular neighborhood $T$ of $t_1$, with a meridian $m(t_1)$.
Consider a $d_1$-torus $T_{d_1}$, a basis $(\alpha,\beta)$ of $H_1(\partial T_{d_1};\mathbb{Z})$, and a curve $t$ in $T_{d_1}$, 
such that $d_1\alpha=0$ and $\beta=d_1t$ in $H_1(T_{d_1};\mathbb{Z})$. Define an LP-surgery $(\frac{T_{d_1}}{T})$ by identifying 
$\alpha$ with $m(t_1)$ and $\beta$ with $t_1$. Set $A'=A(\frac{T_{d_1}}{T})$.
In $A'$, $t_1=d_1 t$, thus we have $\beta_1=d_1\gamma$ with $\gamma=\gamma_1+t$. 

Consider a tunnel $C$ around $\gamma$. Set $B=\overline{A'\setminus C}$. By Lemma \ref{lemma1}, $B$ is a $\mathbb{Q}$HH 
of genus $(g+1)$. There is a surface $S\subset B$ such that 
$\partial S\subset \partial B$ is homologous to $\beta_1 -d_1 \ell+km$ in $\partial B$, where $m$ is a meridian of $\gamma$, $\ell$ is a longitude 
of $\gamma$, and $k$ is an integer. Consider simple closed curves $\sigma_1$ and $\sigma_2$ in $\partial B$
such that $\sigma_1=m-d_1\alpha_1$ and $\sigma_2=\beta_1 -d_1 \ell+k m$ in $H_1(\partial B)$. The curves $\sigma_1$ 
and $\sigma_2$ are null-homologous in $B$, and $(\sigma_1,\sigma_2,\alpha_2,\dots,\alpha_g)$ is a basis of $\mathcal{L}_B^\mathbb{Z}$. 
Hence $|\mathcal{L}_B^T/\mathcal{L}_B^\mathbb{Z}|<|\mathcal{L}_A^T/\mathcal{L}_A^\mathbb{Z}|$.

Consider a genus $(g+1)$ standard handlebody $H_{g+1}$ of boundary $\partial B$, where the $\sigma_i$ and the $\alpha_i$ bound disks 
in $H_{g+1}$. By induction, $B$ can be obtained from $H_{g+1}$ by a finite sequence of elementary surgeries or their inverses. 
Gluing back the cylinder $C$ to $H_{g+1}$, we get a genus $g$ $\mathbb{Q}$HH $\tilde{A}$ satisfying 
$H_1(\tilde{A})=\frac{\mathbb{Z}}{d_1\mathbb{Z}}\alpha_1\oplus\mathbb{Z}\gamma\oplus(\oplus_{2\leq i\leq g}\mathbb{Z}\beta_i)$, 
such that $A'$ can be obtained from $\tilde{A}$ by a finite sequence of elementary surgeries or their inverses. 
Hence $A$ can be obtained from $\tilde{A}$ by a finite sequence of elementary surgeries or their inverses. 
Since $\partial \tilde{A}$ and $\partial H_g$ are both LP-identified with $\partial A$, they are LP-identified with each other. 
By Lemma \ref{lemmacaspart2}, $\tilde{A}$ can be obtained from $H_g$ by a finite sequence of elementary surgeries 
or their inverses. \fin

\paragraph{Remark}
We could have defined elementary surgeries by restricting the genus 1 case to LP-replacements of standard tori by $p$-tori, 
for $p$ prime, and keep Theorem \ref{thelsur} true. Indeed, consider a $d$-torus $T_d$ and the usual curve $\gamma$ in $T_d$ that generates 
$\frac{H_1(T_d;\mathbb{Z})}{Torsion}$. One can check that an LP-replacement of a tubular neighborhood of $\gamma$ 
by a $d'$-torus produces a $dd'$-torus. Hence, for any positive integer $d$, a $d$-torus $T_d$ can be obtained from a standard torus 
by a finite sequence of ``prime'' elementary surgeries of genus 1. Use then the ``tunneling method'' to see that any $d$-torus 
can be obtained from this $T_d$, with the right choice of longitude, by a finite sequence of borromean surgeries.

   \section{Borromean surgeries and clasper calculus} \label{secborro}

Fix a 3-manifold $M$, possibly with boundary. Let $\mathcal{F}_0^{\mathbb{Z}}(M)$ denote the rational vector space 
generated by all the 3-manifolds that can be obtained from $M$ by a finite sequence of borromean surgeries, up 
to orientation-preserving homeomorphism. Let $\mathcal{F}_n^{\mathbb{Z}}(M)$ denote the subspace generated by 
the $[M;\Gamma]$ for all $m$-component Y-link $\Gamma$ in $M$, with $m\geq n$. Let ``$=_n$'' denote the equality modulo $\mathcal{F}_{n+1}^{\mathbb{Z}}(M)$. 

\begin{lemma}[GGP, Corollary 4.3] \label{lemmacut}
Let $\Gamma$ be an $n$-component Y-link in a 3-manifold $M$. Let $\ell$ be a leaf of $\Gamma$. Let $\gamma$ be a framed arc 
starting at the vertex incident to $\ell$ and ending in another point of $\ell$, embedded in $M$ as the core of a band glued to 
the associated surface of $\Gamma$ as shown in Figure \ref{figcut}. The arc $\gamma$ splits the leaf $\ell$ into two leaves 
$\ell'$ and $\ell''$. Denote by $\Gamma'$ and $\Gamma''$ the Y-links obtained from $\Gamma$ by replacing the leaf $\ell$ 
by $\ell'$ and $\ell''$ respectively. Then $[M;\Gamma]=_n[M;\Gamma']+[M;\Gamma'']$.
\end{lemma}
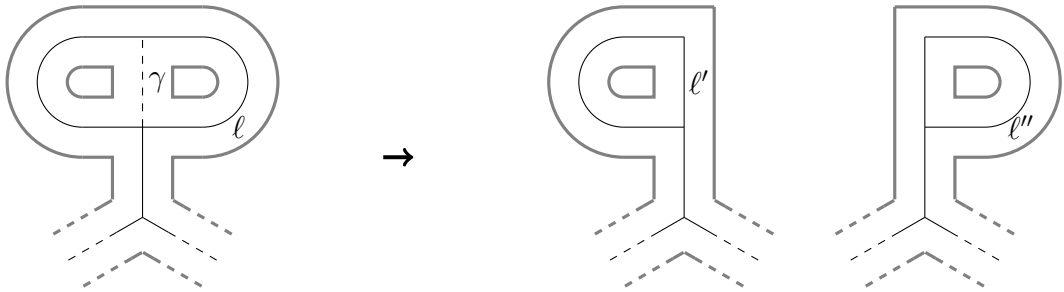
\begin{figure}[htb] 
\begin{center}
\begin{tikzpicture} [scale=0.4]
\newcommand{\edge}[1]{
\draw[rotate=#1] (0,0) -- (0,-1);
\draw[rotate=#1,gray,very thick] (1,-0.6) -- (1,-1);
\draw[rotate=#1,gray,very thick] (-1,-0.6) -- (-1,-1);
\draw[rotate=#1,dashed] (0,-1) -- (0,-3);
\draw[rotate=#1,dashed,gray,very thick] (1,-1) -- (1,-3);
\draw[rotate=#1,dashed,gray,very thick] (-1,-1) -- (-1,-3);}
\edge{60} \edge{-60}
\draw (0,0) -- (0,3);
\draw (-2,3) -- (2,3) (-2,6) -- (2,6);
\draw[dashed] (0,3) -- (0,6);
\draw (0.5,4.5) node{$\gamma$};
\draw (2,3) arc (-90:90:1.5);
\draw (-2,6) arc (90:270:1.5);
\draw (3.2,3) node{$\ell$};
\begin{scope} [gray,very thick]
\draw (2,4) -- (1,4) -- (1,5) -- (2,5);
\draw (-2,4) -- (-1,4) -- (-1,5) -- (-2,5);
\draw (2,4) arc (-90:90:0.5);
\draw (-2,5) arc (90:270:0.5);
\draw (1,0.6) -- (1,2) -- (2,2) (2,7) -- (-2,7);
\draw (-1,0.6) -- (-1,2) -- (-2,2);
\draw (2,2) arc (-90:90:2.5);
\draw (-2,7) arc (90:270:2.5);
\end{scope}
\draw[->,line width=1.5pt] (8,2) -- (9,2);
\begin{scope} [xshift=26cm]
\edge{60} \edge{-60}
\draw (0,0) -- (0,3);
\draw (0,3) -- (2,3) (0,6) -- (2,6);
\draw (0,3) -- (0,6);
\draw (2,3) arc (-90:90:1.5);
\draw (3.2,3) node{$\ell''$};
\begin{scope} [gray,very thick]
\draw (2,4) -- (1,4) -- (1,5) -- (2,5);
\draw (2,4) arc (-90:90:0.5);
\draw (1,0.6) -- (1,2) -- (2,2) (2,7) -- (-1,7);
\draw (-1,0.6) -- (-1,7);
\draw (2,2) arc (-90:90:2.5);
\end{scope}
\end{scope}
\begin{scope} [xshift=18cm]
\edge{60} \edge{-60}
\draw (0,0) -- (0,3);
\draw (-2,3) -- (0,3) (-2,6) -- (0,6);
\draw (0,3) -- (0,6);
\draw (-2,6) arc (90:270:1.5);
\draw (0.5,4.5) node{$\ell'$};
\begin{scope} [gray,very thick]
\draw (-2,4) -- (-1,4) -- (-1,5) -- (-2,5);
\draw (-2,5) arc (90:270:0.5);
\draw (1,0.6) -- (1,7) (1,7) -- (-2,7);
\draw (-1,0.6) -- (-1,2) -- (-2,2);
\draw (-2,7) arc (90:270:2.5);
\end{scope}
\end{scope}
\end{tikzpicture}
\caption{cutting a leaf} \label{figcut}
\end{center}
\end{figure}

\begin{lemma}[GGP, Lemma 4.8] \label{lemmaframing}
Let $\Gamma$ be an $n$-component Y-link in a 3-manifold $M$. If $\Gamma$ has a leaf $\ell$ that bounds a disk in $M\setminus(\Gamma\setminus\ell)$ 
and has framing 1, then $[M;\Gamma]=0$.
\end{lemma}

These two lemmas imply that the class of $[M;\Gamma]$ modulo $\mathcal{F}_{n+1}^{\mathbb{Z}}(M)$ does not depend on the framing 
of the leaves.

\begin{lemma}[GGP, Corollary 4.2] \label{lemmaslide}
Let $\Gamma$ be an $n$-component Y-link in a 3-manifold $M$. Let $K$ be a framed knot in $M\setminus\Gamma$. Let $\Gamma'$ be obtained from $\Gamma$ 
be sliding an edge of $\Gamma$ along $K$ (see Figure \ref{figslide}). Then $[M;\Gamma]=_n[M;\Gamma']$.
\end{lemma}
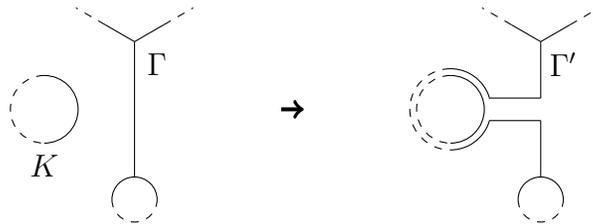
\begin{figure}[htb] 
\begin{center}
\begin{tikzpicture} [scale=0.3]
\newcommand{\edge}[1]{
\draw[rotate=#1] (0,0) -- (0,-1.5);
\draw[rotate=#1,dashed] (0,-1.5) -- (0,-3);}
\edge{120} \edge{240}
\draw (0,0) -- (0,-6);
\draw (1,-7) arc (0:180:1);
\draw[dashed] (-1,-7) arc (180:360:1);
\draw (-4,-4.5) arc (-90:90:1.5);
\draw[dashed] (-4,-1.5) arc (90:270:1.5);
\draw (-4,-5.5) node{$K$};
\draw (1,-1) node{$\Gamma$};
\draw[->,line width=1.5pt] (6.5,-3) -- (7.5,-3);
\begin{scope} [xshift=18cm]
\edge{120} \edge{240}
\draw (0,0) -- (0,-2.5) (0,-3.5) -- (0,-6);
\draw (1,-7) arc (0:180:1);
\draw[dashed] (-1,-7) arc (180:360:1);
\draw (-4,-4.5) arc (-90:90:1.5);
\draw[dashed] (-4,-1.5) arc (90:270:1.5);
\draw (-4,-4.8) arc (-90:90:1.8);
\draw[dashed] (-4,-1.2) arc (90:270:1.8);
\draw[color=white,line width=3pt] (-2.26,-2.53) -- (-2.26,-3.48);
\draw (-2.26,-2.5) -- (0,-2.5) (-2.26,-3.5) -- (0,-3.5);
\draw (1,-1) node{$\Gamma'$};
\end{scope}
\end{tikzpicture}
\caption{sliding an edge} \label{figslide}
\end{center}
\end{figure}

\begin{lemma}[GGP, Lemma 4.4] \label{lemmahalftwist}
 Let $\Gamma$ be an $n$-component Y-link in a 3-manifold $M$. Let $\Gamma'$ be obtained from $\Gamma$ by twisting 
the framing of an edge by a half twist. Then $[M;\Gamma']=_n -[M;\Gamma]$.
\end{lemma}

In the following, we will consider {\em oriented Y-links}, defined as follows. 
A Y-graph is {\em oriented} if its associated surface is oriented. 
An orientation of a Y-graph induces an orientation of its leaves and of its internal vertex, as shown in Figure 
\ref{figorientY}, where the surface drawn is given the standard orientation of the plane. 
\begin{figure}[htb]
 \begin{center}
\begin{tikzpicture} [scale=0.15]
\draw (0,0) circle (1.6);
\draw[->] (1.6,-0.1) -- (1.6,0.1);
\newcommand{\feuille}[1]{
\draw[rotate=#1,line width=4pt,color=white] (0,-1) -- (0,-2);
\draw[rotate=#1,thick,color=gray] (0,-11) circle (5);
\draw[rotate=#1,thick,color=gray] (0,-11) circle (1);
\draw[rotate=#1,line width=8pt,color=white] (-2,-6.42) -- (2,-6.42);
\draw[rotate=#1,thick,color=gray] (2,-1.15) -- (2,-6.42);
\draw[rotate=#1,thick,color=gray] (-2,-1.15) -- (-2,-6.42);
\draw[rotate=#1] (0,0) -- (0,-8);
\draw[rotate=#1] (0,-11) circle (3);
\draw[rotate=#1,->] (-3,-10.9) -- (-3,-11.1);}
\feuille{0}
\feuille{120}
\feuille{-120}
\end{tikzpicture}
 \end{center}
\caption{oriented Y-graph} \label{figorientY}
\end{figure}
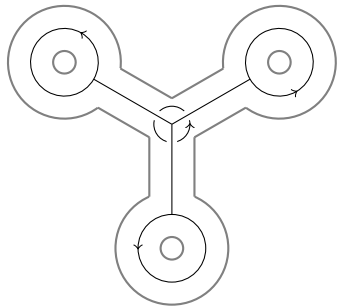
A Y-link is oriented if its components are oriented. In this setting, one can twist the framing of an edge only 
by an integral number of twists. A half twist corresponds to a change of orientation of the adjacent leaf.

Let $\Gamma$ be an oriented Y-link in a 3-manifold $M$.
The above results imply that the class of $[M;\Gamma]$ modulo $\mathcal{F}_{n+1}^\mathbb{Z}(M)$ does not depend on the edges 
of $\Gamma$ and on the incident vertices of the leaves of $\Gamma$. We shall see that, in some sense, it only depends 
on the homology classes of the leaves.

\begin{lemma}[GGP, Lemma 2.2] \label{lemmadisk}
 Let $\Gamma$ be a Y-graph in a 3-manifold $M$, which has a 0-framed leaf $\ell$ that bounds a disk in $M\setminus(\Gamma\setminus\ell)$. 
Then $M(\Gamma)\cong M$.
\end{lemma}

\begin{lemma} \label{lemmatriv}
 Let $\Gamma$ be an oriented $n$-component Y-link in a 3-manifold $M$. Assume $\Gamma$ has a leaf $\ell$ which is trivial 
in $H_1(M\setminus(\Gamma\setminus\ell);\mathbb{Z})$. Then $[M;\Gamma]=_n0$.
\end{lemma}
\begin{proof}
We can assume that $\ell$ is 0-framed. 
 The leaf $\ell$ bounds a surface $\Sigma$ whose interior does not meet $\Gamma$. First assume $\Sigma$ has a positive genus. 
Thanks to Lemma \ref{lemmacut}, we can assume $\Sigma$ has genus 1. Apply Lemma \ref{lemmacut} to decompose $\ell$ into four leaves, 
and apply it again to re-glue them by pairs, as shown in Figure \ref{figdec}. 
\begin{figure}[htb]
\begin{center}
\begin{tikzpicture} [scale=0.2]
\draw (0,-3) -- (0,0);
\draw (-9,4) -- (-9,0) -- (9,0) -- (9,4);
\draw (-7,4) -- (-3,4) (-1,4) -- (1,4) (3,4) -- (7,4);
\draw (7,4) arc (0:180:4);
\draw (9,4) arc (0:180:6);
\draw[line width=11,color=white] (2,4) arc (0:180:5);
\draw (1,4) arc (0:180:4);
\draw (3,4) arc (0:180:6);
\draw[dashed] (-3,4) -- (0,0) -- (1,4) (0,0) -- (3,4);
\draw (7.5,1.5) node{$\Sigma$};
\begin{scope} [xshift=35cm]
\draw (0,-3) -- (0,0) (-9,0) -- (-9,4);
\draw (-7,4) -- (-3,4);
\draw (1,4) arc (0:180:4);
\draw (3,4) arc (0:180:6);
\draw (1,4) .. controls +(-1,-4) and +(-3,-4) .. (3,4);
\draw (-9,0) -- (0,0) -- (-3,4);
\end{scope}
\draw[->,line width=2pt] (16,2) -- (18,2);
\begin{scope} [xshift=45cm]
\draw (0,-3) -- (0,0) -- (9,0) -- (9,4);
\draw (-1,4) -- (1,4) (3,4) -- (7,4);
\draw (7,4) arc (0:180:4);
\draw (9,4) arc (0:180:6);
\draw (0,0) -- (3,4);
\draw (-3,4) .. controls +(3,-4) and +(-1,-4) .. (1,4);
\end{scope}
\end{tikzpicture}
\end{center}
\caption{decomposing a leaf} \label{figdec}
\end{figure}
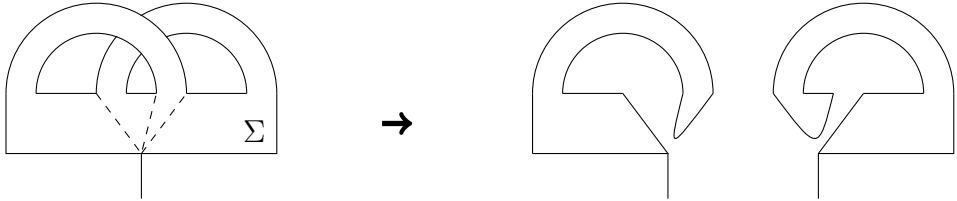
This leads us to the case of a leaf which bounds a disk. The result follows then from Lemma \ref{lemmadisk}. 
\end{proof}

\begin{lemma} \label{lemmah}
 Let $\Gamma$ be an $n$-component Y-link in a 3-manifold $M$. Let $\ell$ be a leaf of $\Gamma$. Fix $\Gamma\setminus\ell$. 
Then the class of $[M;\Gamma]\ mod\ \mathcal{F}_{n+1}^\mathbb{Z}(M)$ is a linear function 
of $\ell\in H_1(M\setminus(\Gamma\setminus\ell);\mathbb{Q})$.
\end{lemma}
\begin{proof}
 Consider an $n$-component Y-link $\Gamma'$ that has a leaf $\ell'$ such that $\Gamma'\setminus\ell'$ coincides with 
$\Gamma\setminus\ell$ and $\ell'$ is homologous to $\ell$ in $M\setminus(\Gamma\setminus\ell)$. Construct another 
$n$-component Y-link $\Gamma^\delta$ by replacing the leaf $\ell$ by $\ell-\ell'$ in $\Gamma$ (see Figure \ref{figll'}). 
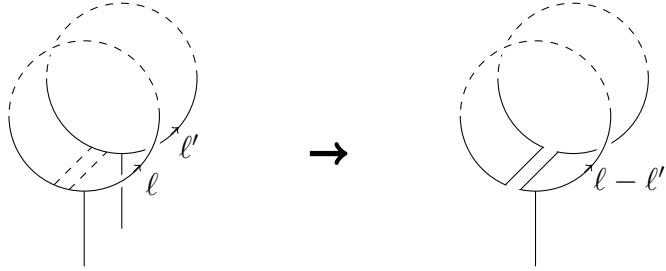
\begin{figure}[htb] 
\begin{center}
\begin{tikzpicture} [scale=0.5]
\newcommand{\leaf}[3]{
\draw[xshift=#1,yshift=#2] (0,0) -- (0,-2);
\draw[xshift=#1,yshift=#2] (-2,2) arc (-180:0:2);
\draw[xshift=#1,yshift=#2,dashed] (2,2) arc (0:180:2);
\draw[xshift=#1,yshift=#2,->] (1.3,0.5) -- (1.5,0.7);
\draw[xshift=#1,yshift=#2] (1.8,0.2) node{#3};}
\leaf{1cm}{1cm}{$\ell'$}
\draw[line width=4pt,color=white] (0,2) circle (2);
\leaf{0}{0}{$\ell$}
\draw[dashed] (-0.4,0.04) -- (0.6,1.04) (-0.8,0.17) -- (0.2,1.17);
\draw[->,line width=2pt] (6,1) -- (7,1);
\begin{scope} [xshift=12cm]
\newcommand{\sleaf}[2]{
\draw[xshift=#1,yshift=#2] (-2,2) arc (-180:0:2);
\draw[xshift=#1,yshift=#2,dashed] (2,2) arc (0:180:2);}
\sleaf{1cm}{1cm}
\draw[line width=4pt,color=white] (0,2) circle (2);
\sleaf{0}{0}
\draw (0,0) -- (0,-2);
\draw[->] (1.3,0.5) -- (1.5,0.7);
\draw (2.5,0.3) node{$\ell-\ell'$};
\draw (-0.4,0.04) -- (0.6,1.04) (-0.8,0.17) -- (0.2,1.17);
\draw[line width=4.5pt,color=white] (-0.7,0) -- (0.5,1.21);
\end{scope}
\end{tikzpicture}
\end{center}
\caption{the leaf $\ell-\ell'$} \label{figll'}
\end{figure}
By Lemma \ref{lemmatriv}, 
$[M;\Gamma^\delta]=0$. Thus Lemma \ref{lemmacut} implies $[M;\Gamma]=_n[M;\Gamma']$. Hence, for $\Gamma\setminus\ell$ fixed, 
$[M;\Gamma]\ mod\ \mathcal{F}_{n+1}^\mathbb{Z}(M)$ only depends on the class of $\ell$ in $H_1(M\setminus(\Gamma\setminus\ell);\mathbb{Z})$. 
The linearity follows from Lemma \ref{lemmacut}. Since the $\mathcal{F}_n^\mathbb{Z}(M)$ are rational vector spaces, 
$[M;\Gamma]\ mod\ \mathcal{F}_{n+1}^\mathbb{Z}(M)$ only depends on the rational homology class of $\ell$.
\end{proof}

Now, in the case of $\mathbb{Q}$HS's, we want to restrict the set of generators of 
$\mathcal{F}_n^\mathbb{Z}(M)/\mathcal{F}_{n+1}^\mathbb{Z}(M)$ to brackets defined by Jacobi diagrams.

Lemma \ref{lemmaslide} implies:
\begin{lemma} \label{lemmaindJacobi}
Let $J$ be a Jacobi diagram of degree $\frac{n}{2}$. Equip $J$ with a framing induced by an immersion of $J$ in the plane. 
Embed the framed diagram $J$ in a 3-manifold $M$. Let $\Gamma$ be the oriented $n$-component Y-link obtained from $J$ by replacing 
its edges as shown in Figure \ref{figrempla}. Then the class of $[M;\Gamma]$ modulo $\mathcal{F}_{n+1}^\mathbb{Z}(M)$ does not depend 
on the embedding and framing of $J$.
\end{lemma}
\begin{figure}[htb] 
\begin{center}
\begin{tikzpicture} [scale=0.2]
\draw (-36,0) -- (-18,0);
\draw (-36,0) node{$\scriptscriptstyle{\bullet}$};
\draw (-18,0) node{$\scriptscriptstyle{\bullet}$};
\draw[->,line width=1.5pt,>=latex] (-10.5,0) -- (-7.5,0);
\draw (0,0) node{$\scriptscriptstyle{\bullet}$};
\draw (0,0) -- (6,0);
\draw (8,0) circle (2);
\draw[color=white,line width=6pt] (8,0) arc (-180:-90:2);
\draw (10,0) circle (2);
\draw[color=white,line width=6pt] (10,0) arc (0:90:2);
\draw (10,0) arc (0:90:2);
\draw[->] (11.5,1.3) -- (11.3,1.5);
\draw[->] (6.5,-1.3) -- (6.7,-1.5);
\draw (12,0) -- (18,0);
\draw (18,0) node{$\scriptscriptstyle{\bullet}$};
\end{tikzpicture}
\end{center}
\caption{Replacement of an edge} \label{figrempla}
\end{figure}
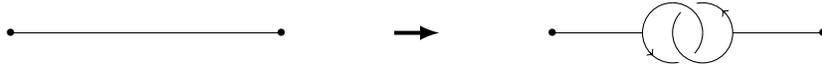 

In the sequel, we will denote by $[M;J]$ the class of $[M;\Gamma]$ modulo $\mathcal{F}_{n+1}^\mathbb{Z}(M)$.

\begin{lemma} \label{lemmasimplify}
Let $\Gamma$ be an oriented $n$-component Y-link in a 3-manifold $M$. Assume that all the leaves of $\Gamma$ are trivial 
in $H_1(M;\mathbb{Q})$. Then $[M;\Gamma]$ is equal to a $\mathbb{Q}$-linear combination 
of terms $[M;J]$ for some Jacobi diagrams $J$, modulo $\mathcal{F}_{n+1}^\mathbb{Z}(M)$.
\end{lemma}
\begin{proof}
 Suppose $\Gamma$ has a leaf $\ell$ which is non trivial in $H_1(M;\mathbb{Z})$. Then there is a positive integer $k$ 
such that $k\ell=0$ in $H_1(M;\mathbb{Z})$. Denote by $\Gamma'$ the Y-link obtained from $\Gamma$ by replacing 
the leaf $\ell$ by a leaf homologous to $k\ell$ in $H_1(M\setminus(\Gamma\setminus\ell);\mathbb{Z})$. 
By Lemma \ref{lemmah}, we have $[M;\Gamma]=_n\frac{1}{k}[M;\Gamma']$. 
Thus we can assume that all the leaves of $\Gamma$ are null-homologous in $M$. As we have seen above, we also can assume 
that they are 0-framed. 

Such leaves bound embedded surfaces in $M$. Thanks to Lemma \ref{lemmaslide}, we can assume that the interior of these surfaces 
do not meet the edges of $\Gamma$. Consider a leaf $\ell$ of $\Gamma$. Apply Lemma \ref{lemmacut} to cut $\ell$ into some leaves which are 
meridians of other leaves, and one leaf which bounds a surface in $M\setminus(\Gamma\setminus \ell)$. The last one can be excluded 
by applying Lemma \ref{lemmatriv}. Cutting similarly each leaf of $\Gamma$, we obtain Y-links whose leaves 
are linked by pairs, in the pattern of Hopf links. Since Lemma \ref{lemmah} allows us to change the orientation of a leaf, 
modulo a sign, we get Y-links obtained from Jacobi diagrams.
\end{proof}

\begin{corollary} \label{corJacobi}
 Let $M$ be a $\mathbb{Q}$HS. Then $\frac{\mathcal{F}_n^{\mathbb{Z}}(M)}{\mathcal{F}_{n+1}^{\mathbb{Z}}(M)}$ 
is generated by the $[M;J]$ for all Jacobi diagrams $J$ of degree $\frac{n}{2}$. 
In particular, if $n$ is odd, $\mathcal{F}_n^{\mathbb{Z}}(M)=\mathcal{F}_{n+1}^{\mathbb{Z}}(M)$.
\end{corollary}

We end the section by focusing the case of Y-graphs.
\begin{lemma} \label{lemmaa}
 Let $\Gamma$ be an oriented Y-graph in a 3-manifold $M$. Suppose that $\Gamma$ has two leaves $\ell$ and $\ell'$ that bound disks 
in $M\setminus(\Gamma\setminus(\ell\cup\ell'))$ and that form a positive Hopf link. Then $[M;\Gamma]=_2 0$.
\end{lemma}
\begin{proof}
If the curve $\alpha$ obtained from the leaves $\ell$ and $\ell'$ and their adjacent edges, as shown in Figure \ref{fighopf}, is 0-framed 
and bounds a disk whose interior does not meet $\Gamma$, then, acoording to \cite[Lemma 2.3]{GGP}, the surgery on $\Gamma$ preserves 
the homeomorphism class of $M$. 
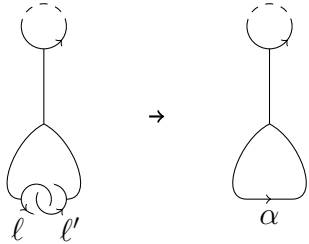
\begin{figure}[htb] 
\begin{center}
\begin{tikzpicture} [scale=0.2]
\draw (0,0) .. controls +(2,-1) and +(2,0) .. (1.5,-5);
\draw (0,0) .. controls +(-2,-1) and +(-2,0) .. (-1.5,-5);
\draw (0,0) -- (0,5);
\draw[dashed] (1.5,6.5) arc (0:180:1.5);
\draw (-1.5,6.5) arc (-180:0:1.5);
\draw[->] (1.1,5.5) -- (1.2,5.6);
\begin{scope} [xshift=-4.5cm,yshift=-5cm,scale=0.5]
\draw (8,0) circle (2);
\draw[color=white,line width=6pt] (8,0) arc (-180:-90:2);
\draw (10,0) circle (2);
\draw[color=white,line width=6pt] (10,0) arc (0:90:2);
\draw (10,0) arc (0:90:2);
\draw[<-] (11.5,-1.3) -- (11.3,-1.5);
\draw (12.5,-4) node{$\ell'$};
\draw[->] (6.5,-1.3) -- (6.7,-1.5);
\draw (5.5,-4) node{$\ell$};
\end{scope}
\draw[->,thick] (7,0.5) -- (8,0.5);
\begin{scope} [xshift=15cm]
\draw (0,0) .. controls +(2,-1) and +(2,0) .. (1.5,-5);
\draw (0,0) .. controls +(-2,-1) and +(-2,0) .. (-1.5,-5);
\draw (0,0) -- (0,5);
\draw[dashed] (1.5,6.5) arc (0:180:1.5);
\draw (-1.5,6.5) arc (-180:0:1.5);
\draw[->] (1.1,5.5) -- (1.2,5.6);
\draw (1.5,-5) -- (-1.5,-5);
\draw[->] (0,-5) -- (0.1,-5);
\draw (0,-5) node[below] {$\alpha$};
\end{scope}
\end{tikzpicture}
\end{center}
\caption{The Y-graph $\Gamma$ and the associated curve $\alpha$} \label{fighopf}
\end{figure}
Lemma \ref{lemmaslide} allows us to reduce the proof to this case. 
\end{proof}

\begin{lemma} \label{lemmab}
 Let $\Gamma$ be an oriented Y-graph in a 3-manifold $M$. If $\Gamma$ has a leaf $\ell$ which is trivial in $H_1(M;\mathbb{Q})$, 
then $[M;\Gamma]=_2 0$.
\end{lemma}
\begin{proof}
 As in the proof of Lemma \ref{lemmasimplify}, we can assume that $\ell$ is null-homologous in $M$ and 0-framed. 
Then $\ell$ bounds a surface $\Sigma$. Using Lemma \ref{lemmaslide}, we can assume that its interior $\mathring{\Sigma}$ does not meet
the edges of $\Gamma$. However, it can meet the other leaves. Using Lemma \ref{lemmacut} to decompose the different leaves 
of $\Gamma$, we can restrict to two cases. Either $\mathring{\Sigma}$ does not meet $\Gamma$, 
or $\ell$ is linked with another leaf in the pattern of a Hopf link. Conclude with Lemma \ref{lemmatriv} in the first case. 
In the second case, since Lemma \ref{lemmah} allows us to change the orientation of a leaf, modulo a sign, conclude with 
Lemma \ref{lemmaa}.
\end{proof}

\begin{lemma} \label{lemmaF1}
 Let $\Gamma$ be an oriented Y-graph in a 3-manifold $M$. The class of $[M;\Gamma]$ modulo $\mathcal{F}_2^\mathbb{Z}(M)$ 
only depends on the classes of the leaves of $\Gamma$ in $H_1(M;\mathbb{Q})$. Moreover, the dependance is trilinear 
and alternating.
\end{lemma}
\begin{proof}
 Consider a leaf $\ell$ of $\Gamma$. Consider an oriented Y-graph $\Gamma'$ and a leaf $\ell'$ of $\Gamma'$ such that 
$\Gamma'\setminus\ell'$ coincides with $\Gamma\setminus\ell$ and $\ell'=\ell$ in $H_1(M;\mathbb{Q})$. Construct another 
Y-graph $\Gamma^\delta$ by replacing the leaf $\ell$ by $\ell-\ell'$ in $\Gamma$ (see Figure \ref{figll'}). 
By Lemma \ref{lemmab}, 
$[M;\Gamma^\delta]=0$. Thus Lemma \ref{lemmacut} implies $[M;\Gamma]=_n[M;\Gamma']$. Hence, for $\Gamma\setminus\ell$ fixed, 
$[M;\Gamma]\ mod\ \mathcal{F}_2^\mathbb{Z}(M)$ only depends on the class of $\ell$ in $H_1(M;\mathbb{Q})$. 
The linearity follows from Lemma \ref{lemmacut}. To get the alternating property, note that exchanging two leaves is 
equivalent to changing the orientation of the three leaves.
\end{proof}

  \section{Finite type invariants of degree 1} \label{secdeg1}

    \subsection{The family $(M_p-S^3)_{p\ prime}$ generates $\mathcal{G}_1$} \label{subsecgene}

We denote by ``$=_2$'' the equality modulo $\mathcal{F}_2$. Note that $\mathcal{F}_1$ is generated by the \mbox{$(M-S^3)$.}
For any $\mathbb{Q}$HS $M$, let 
$lk_M:H_1(M;\mathbb{Z})\times H_1(M;\mathbb{Z})\to\frac{\mathbb{Q}}{\mathbb{Z}}$
be the linking form on $H_1(M;\mathbb{Z})$.

\begin{lemma} \label{lemmalink}
Let $M$ and $N$ be $\mathbb{Q}$HS's such that $(H_1(M;\mathbb{Z}),lk_M)\cong (H_1(N;\mathbb{Z}),lk_N)$. Then $M=_2 N$.
\end{lemma}
\begin{proof}
By \cite[Theorem 2]{Mat}, $N$ can be obtained from $M$ by a finite sequence of borromean surgeries.
It suffices to show that $M(\frac{B'}{B})=_2M$ for one borromean surgery $(\frac{B'}{B})$.
This follows from Corollary \ref{corJacobi}.
\end{proof}

Like in \cite{KK}, we call \emph{linking} a pair $(H,\phi)$, where $H$ is a finite abelian group, 
and $\phi$ is a non degenerate symmetric bilinear form on $H$, with values in $\mathbb{Q} / \mathbb{Z}$.
Consider the abelian semigroup $\mathfrak{N}$ of all linkings under orthogonal sum. We have a homomorphism 
$\mathcal{H}$ from the semigroup of all $\mathbb{Q}$HS's under connected sum to $\mathfrak{N}$, given by
$\mathcal{H}(M)=(H_1(M;\mathbb{Z}),lk_M)$. By \cite[Theorem 6.1]{KK}, this homomorphism is onto. So we can define 
an equivalence relation on $\mathfrak{N}$ by $H_1\sim_2 H_2$ if $H_1=\mathcal{H}(M_1)$, $H_2=\mathcal{H}(M_2)$,
and $M_1-S^3=_2 M_2-S^3$.

Note that: $$\hspace{5cm} M\sharp N-S^3=_2 (M-S^3)+(N-S^3).\hspace{4cm} (\star)$$ 
Thus, by Lemma \ref{lemmalink}, in order to prove that 
$(M_p-S^3)_{p\ prime}$ generates $\mathcal{G}_1$, it suffices to show that 
any $H\in\mathfrak{N}$ is 2-equivalent to a direct sum of groups $\mathbb{Z}_p:=\mathbb{Z} / p\mathbb{Z}$, with $p$ prime, 
independently of the associated bilinear form. Since $\mathfrak{N}$ is the direct sum of the abelian semigroups 
$\mathfrak{N}_p$ of linkings on $p$-groups, we restrict ourselves to the study of $p$-groups.

\begin{lemma} \label{lemmalinkings}
 Any linking in $\mathfrak{N}_p$ is 2-equivalent to an orthogonal sum of linkings on cyclic $p$-groups. 
Two linkings defined on the same cyclic group are 2-equivalent.
\end{lemma}
\begin{proof}
In the case of odd primes $p$, by \cite[Theorem 4]{Wall}, $\mathfrak{N}_p$ has generators
$A_{p^k}$, $B_{p^k}$, $k\geq 1$, and sole relation $2A_{p^k}=2B_{p^k}\ (\mathcal{R}_{p^k})$, where:
\begin{itemize}
 \item $A_{p^k}=(\mathbb{Z}_{p^k}, \phi_A)$, $\phi_A(1,1)=\frac{1}{p^k}$,
 \item $B_{p^k}=(\mathbb{Z}_{p^k}, \phi_B)$, $\phi_B(1,1)=\frac{x}{p^k}$, with $x$ non square modulo $p^k$.
\end{itemize}
The relations $(\star)$ and $(\mathcal{R}_{p^k})$ show that $A_{p^k}\sim_2 B_{p^k}$.

In the case of 2-groups, we will use the presentation of $\mathfrak{N}_2$ given in the introduction of 
\cite{Mir} by Miranda, who gives an alternative version of the description of $\mathfrak{N}_2$ obtained by Kawauchi and Kojima 
in \cite{KK}. The generators are 4 linkings $A_k$, $B_k$, $C_k$, $D_k$, defined on $\mathbb{Z}_{2^k}$,
and 2 linkings $E_k$, $F_k$, defined on $\mathbb{Z}_{2^k}\times\mathbb{Z}_{2^k}$.
The relation $A_k+E_k=2A_k+B_k$ implies $E_k\sim_2 A_k+B_k$, and the relation $2E_k=2F_k$ implies 
$E_k\sim_2 F_k$. So we're lead to the cyclic case. The relations $2A_k=2C_k$, $2B_k=2D_k$, $4A_k=4B_k$,
give $A_k\sim_2 B_k\sim_2 C_k\sim_2 D_k$.\end{proof}

Lemma \ref{lemmalinkings} reduces our study to the case of cyclic groups with arbitrary linkings.
\begin{lemma} \label{lemmagroup}
 Denote by $G_{p^k}$ any linking on $\mathbb{Z}_{p^k}$.
We have $G_{p^{k+k'}}\sim_2 G_{p^k} + G_{p^{k'}}$ for any prime $p$ and any positive integers $k$ and $k'$.
It follows that $G_{p^k}\sim_2 k\, G_p$.
\end{lemma}
\begin{proof} We will use the following easy result.
\begin{sublemma} \label{sublemmasur}
 Let $d$ be a positive integer. Let $T_d$ be a $d$-torus. Let $(\alpha,\beta)$ be a symplectic basis of $H_1(\partial T_d;\mathbb{Z})$ 
such that $\alpha$ generates $\mathcal{L}_{T_d}$ and $\beta=d\gamma$ in $H_1(T_d;\mathbb{Z})$.
Let $T$ be a standard solid torus trivially embedded in $S^3$. Define an LP-identification $\partial T_d\cong\partial T$ that 
identifies $\beta$ with the preferred longitude of $T$. Then 
$H_1(S^3(\frac{T_d}{T}))=\mathbb{Z}_d\alpha\oplus\mathbb{Z}_d\gamma$.
\end{sublemma}
In $S^3$, consider two disjoint, trivially embedded, tori $T$ and $T'$, linked in the pattern of a Hopf link. 
Consider the LP-surgeries given by Sublemma \ref{sublemmasur} for $d=p^k$ and for $d=p^{k'}$. 
We still denote by $\alpha$, $\beta$, $\gamma$ (resp. $\alpha'$, $\beta'$, $\gamma'$) the curves defined in the lemma.
We have $H_1(S^3(\frac{T_{p^k}}{T}))=\mathbb{Z}_{p^k}\alpha\times\mathbb{Z}_{p^k}\gamma$ and 
$H_1(S^3(\frac{T_{p^{k'}}}{T'}))=\mathbb{Z}_{p^{k'}}\alpha'\times\mathbb{Z}_{p^{k'}}\gamma'$.
Now, in $S^3(\frac{T_{p^k}}{T},\frac{T_{p^{k'}}}{T'})$, we have $\alpha'=\beta=p^k\gamma$ and $\alpha=\beta'=p^{k'}\gamma'$. 
Thus $H_1(S^3(\frac{T_{p^k}}{T},\frac{T_{p^{k'}}}{T'}))=\mathbb{Z}_{p^{k+k'}}\gamma\times\mathbb{Z}_{p^{k+k'}}\gamma'$. 
Conclude with the following equality: 
$S^3(\frac{T_{p^k}}{T},\frac{T_{p^{k'}}}{T'})-S^3=_2 (S^3(\frac{T_{p^k}}{T})-S^3)+(S^3(\frac{T_{p^{k'}}}{T'})-S^3)$.\end{proof}

This achieves the proof of the first part of Proposition \ref{propdeg1}, namely the fact that the family $(M_p-S^3)_{p\ prime}$ 
generates $\mathcal{G}_1$.

    \subsection{The invariants $\nu_p$} \label{subsecinv}

In this subsection, unless otherwise mentioned, all the homology modules are considered with integral coefficients.
We prove the following proposition that implies Proposition \ref{propinv}. 

\begin{proposition}
 Consider a $\mathbb{Q}$HS $M$, two disjoint genus $g$ $\mathbb{Q}$HH's $A$ and $B$ in $M$, and two $\mathbb{Q}$HH's $A'$ and $B'$ 
whose boundaries are LP-identified with $\partial A$ and $\partial B$ respectively. Then:
$$\frac{|H_1(M)|}{|H_1(M(\frac{A'}{A}))|}=\frac{|H_1(M(\frac{B'}{B}))|}{|H_1(M(\frac{A'}{A},\frac{B'}{B}))|}.$$
\end{proposition}
\begin{proof}
The exact sequence associated with $(M,A)$ gives:
$$0\to H_2(M,A)\to H_1(A) \to H_1(M) \to H_1(M,A) \to 0.$$
Set $X=M\setminus Int(A)$. By excision, we have $H_i(M,A)=H_i(X,\partial X)$ for any integer $i$.
So the above exact sequence can be rewritten as follows.
$$0\to H_2(X,\partial X)\fl{\varphi_1} H_1(A) \fl{\varphi_2} H_1(M) \fl{\varphi_3} H_1(X,\partial X) \to 0$$
Since $H_1(M)$ is finite, $H_1(X,\partial X)$ also is, and we have $|H_1(M)|=|H_1(X,\partial X)|.|Im(\varphi_2)|$.

Similarly, we have an exact sequence:
$$0\to H_2(X,\partial X)\fl{\varphi'_1} H_1(A') \fl{\varphi'_2} H_1(M(\frac{A'}{A}))\fl{\varphi'_3} H_1(X,\partial X)\to 0.$$
We get: $$\frac{|H_1(M)|}{|H_1(M(\frac{A'}{A}))|}=\frac{|Im(\varphi_2)|}{|Im(\varphi'_2)|}.$$

Similarly arguing with $M(\frac{B'}{B})$ instead of $M$, and setting $X'=X(\frac{B'}{B})$, we have the exact sequences:
$$0\to H_2(X',\partial X')\fl{\psi_1} H_1(A) \fl{\psi_2} H_1(M(\frac{B'}{B})) \fl{\psi_3} H_1(X',\partial X') \to 0,$$
$$0\to H_2(X',\partial X')\fl{\psi'_1} H_1(A')\fl{\psi'_2} H_1(M(\frac{A'}{A},\frac{B'}{B}))\fl{\psi'_3}H_1(X',\partial X')\to 0,$$
and we get:
$$\frac{|H_1(M(\frac{B'}{B}))|}{|H_1(M(\frac{A'}{A},\frac{B'}{B}))|}=\frac{|Im(\psi_2)|}{|Im(\psi'_2)|}.$$

We now relate $|Im(\varphi_2)|$ and $|Im(\psi_2)|$. Since $Im(\varphi_2)\cong\frac{H_1(A)}{Im(\varphi_1)}$ 
and $Im(\psi_2)\cong\frac{H_1(A)}{Im(\psi_1)}$, we shall study $Im(\varphi_1)$ and $Im(\psi_1)$.

The following sublemma gives us additional information about $X$.
\begin{sublemma} \label{sub1}
 If $M$ is a $\mathbb{Q}$HS and if $A$ is a genus $g$ $\mathbb{Q}$HH in $M$, then $X=M\setminus Int(A)$ also is 
a genus $g$ $\mathbb{Q}$HH.
\end{sublemma}
\begin{proof}
It is clear that $H_3(X;\mathbb{Q})=0$ and $H_0(X;\mathbb{Q})=\mathbb{Q}$.

The Mayer-Vietoris sequence associated with $M=A\cup X$ gives:
$$0\to H_3(M;\mathbb{Q})\to H_2(\partial A;\mathbb{Q}) \to H_2(A;\mathbb{Q})\oplus H_2(X;\mathbb{Q}) \to 0.$$
Since $H_3(M;\mathbb{Q})\to H_2(\partial A;\mathbb{Q})$ is an isomorphism that identifies the fundamental classes, 
we have $H_2(X;\mathbb{Q})=0$.

The Mayer-Vietoris sequence also gives an isomorphism $H_1(\partial A;\mathbb{Q}) \cong H_1(A;\mathbb{Q}) \oplus H_1(X;\mathbb{Q})$, 
thus $H_1(X;\mathbb{Q})=\mathbb{Q}^g$.\end{proof}

We have the following commutative diagram, where $i_\star$ is the map induced by the inclusion $i: \partial A \hookrightarrow A$.
\begin{center}
\begin{tikzpicture}
\draw (-0.6,1) node {$H_2(X,\partial X)$};
\draw [->] (0.7,1) -- (1.3,1);
\draw (1,1.3) node {$\partial$};
\draw (2.3,1) node {$H_1(\partial A)$};
\draw [->] (0.7,0.6) -- (1.3,0.2);
\draw (0.9,0.2) node {$\varphi_1$};
\draw [->] (2.3,0.6) -- (2.3,0.2);
\draw (2.6,0.4) node {$i_\star$};
\draw (2.3,-0.2) node {$H_1(A)$};
\draw (5.2,1) node {$H_2(X',\partial X')$};
\draw [->] (3.9,1) -- (3.3,1);
\draw (3.6,1.3) node {$\partial$};
\draw [->] (3.9,0.6) -- (3.3,0.2);
\draw (3.8,0.2) node {$\psi_1$};
\end{tikzpicture}
\end{center}
Denote the images of $H_2(X,\partial X)$ and $H_2(X',\partial X')$ in $H_1(\partial A)$ by $F$ and $F'$ respectively.
Since $\varphi_1$ and $\psi_1$ are injective, the two boundary operators also are.
Thus, by Sublemma \ref{sub1} and Lemma \ref{lemmaprel}, $F$ and $F'$ are free submodules of $H_1(\partial A)$, of rank $g$.
Consider bases $\gamma$ of $F$ and $\gamma'$ of $F'$. Over $\mathbb{Q}$, $F$ generates the Lagrangian $\mathcal{L}_X$,
and $F'$ generates $\mathcal{L}_{X'}$. Since $X'$ is obtained from $X$ by an LP-surgery, we have $\mathcal{L}_X=\mathcal{L}_{X'}$.
Hence we have a matrix $R\in GL_g(\mathbb{Q})$ of change of basis from $\gamma$ to $\gamma'$.
Thus:
$$|Im(\psi_2)|=|\frac{H_1(A)}{Im(\psi_1)}|=|\frac{H_1(A)}{i_\star(F')}|=|\det(R)|.|\frac{H_1(A)}{i_\star(F)}|
=|\det(R)|.|Im(\varphi_2)|.$$

Since the same submodules $F$ and $F'$ occur in the decomposition of $\varphi'_1$ and $\psi'_1$, we also have: 
$$|Im(\psi'_2)|=|\det(R)|.|Im(\varphi'_2)|.$$

Finally,
$$\frac{|H_1(M(\frac{B'}{B}))|}{|H_1(M(\frac{A'}{A},\frac{B'}{B}))|}=\frac{|Im(\psi_2)|}{|Im(\psi'_2)|}
=\frac{|\det(R)|.|Im(\varphi_2)|}{|\det(R)|.|Im(\varphi'_2)|}=\frac{|H_1(M)|}{|H_1(M(\frac{A'}{A}))|}.$$\end{proof}

    \section{Additive invariants of degree $n>1$} \label{secadd}

      \subsection{Degree 1 invariants of framed rational homology tori} \label{subsectori}

Fix a genus 1 surface $\Sigma_1$ and a symplectic basis $(\alpha_0,\beta_0)$ of $H_1(\Sigma_1;\mathbb{Z})$. Define $\mathcal{F}_0(\Sigma_1)$ 
\linebreak[4] as the rational vector space generated by all the rational homology tori $T$, equipped \linebreak[4] 
with an oriented longitude $\ell(T)$. Denote by $m(T)$ the meridian of $T$ that satisfies \linebreak[4] $<m(T),\ell(T)>_{\partial T}=1$. 
The data of the framing is equivalent to the data of an orientation-preserving homeomorphism $h: \Sigma_1 \to \partial T$ such that 
$h_*(\mathbb{Q}\alpha_0)=\mathcal{L}_T$, the equivalence being given by $m(T)=h(\alpha_0)$ and $\ell(T)=h(\beta_0)$. 
In particular, given two framed rational homology tori, we have a canonical 
LP-identification of their boundaries, which identifies the fixed longitudes.
Define a filtration $(\mathcal{F}_n(\Sigma_1))_{n\in\mathbb{N}}$, 
and quotients $(\mathcal{G}_n(\Sigma_1))_{n\in\mathbb{N}}$, as in the case of $\mathbb{Q}$HS's. 
Note that $\mathcal{G}_0(\Sigma_1)\cong\mathbb{Q}$. 

Denote by $T_0$ the standard solid torus with a fixed longitude $\ell(T_0)$. 
For any prime $p$, fix a $\mathbb{Q}$HS $M_p$ such that $H_1(M_p;\mathbb{Z})\cong\mathbb{Z}/p\mathbb{Z}$. Define a
rational homology ball $B_p$ by removing an open ball from $M_p$. In this subsection, we prove:
\begin{proposition} \label{propinvtori}
 $\displaystyle \ \mathcal{G}_1(\Sigma_1)=\bigoplus_{p\ prime} \mathbb{Q} [T_0;\frac{B_p}{B^3}]$
\end{proposition}

 Consider a framed rational homology torus $T$. Set $d(T)=|\mathcal{L}_T^T/\mathcal{L}_T^\mathbb{Z}|$ (see Lemma \ref{lemmaduality} 
for the definition of $\mathcal{L}_T^T$ and $\mathcal{L}_T^\mathbb{Z}$). For $p$ prime,
define $\mu_p(T)=v_p(d(T)\, |Tors(H_1(T;\mathbb{Z}))|)$, where $v_p$ denotes the $p$-adic valuation.
\begin{lemma}
 For any prime $p$, $\mu_p$ is a degree 1 invariant of the framed rational homology tori.
\end{lemma}
\begin{proof}
 Consider a framed rational homology torus $T$. Define a $\mathbb{Q}$HS $M(T)$ by gluing $T$ and the standard torus $T_0$ along
their boundaries, in such a way that $\ell(T)$ is identified with $m(T_0)$.
We have $H_1(M(T))=\frac{H_1(T)}{\mathbb{Z}\ell(T)}$. By Lemma \ref{lemmaduality}, $|H_1(M(T))|=d(T)\, |Tors(H_1(T))|$.
Thus $\mu_p(T)=v_p(|H_1(M(T))|)=\nu_p(M(T))$. The result follows from the fact that $\nu_p$ is a 
degree 1 invariant of $\mathbb{Q}$HS's.
\end{proof}
\begin{corollary} \label{cordirect}
 The sum $\bigoplus_{p\ prime} \mathbb{Q} [T_0;\frac{B_p}{B^3}]$ is direct.
\end{corollary}

\begin{lemma} \label{lemmaredtori}
The space $\mathcal{G}_1(\Sigma_1)$ is generated by the $[T;\frac{E'}{E}]$, where $(\frac{E'}{E})$ is an elementary surgery.
\end{lemma}
\begin{proof}
Consider $[T;\frac{A'}{A}]\in \mathcal{F}_1(\Sigma_1)$. By Theorem \ref{thelsur},
$A'$ is obtained from $A$ by a sequence of elementary surgeries, or their inverses, $(\frac{E_i'}{E_i})_{1\leq i\leq k}$.
Set $A_i=A(\frac{E_1'}{E_1})(\frac{E_2'}{E_2})..(\frac{E_i'}{E_i})$. Then:
$$[T;\frac{A'}{A}]=
\sum_{i=0}^{k-1} [T(\frac{A_i}{A});\frac{A_{i+1}}{A_i}]=\sum_{i=0}^{k-1} [T(\frac{A_i}{A});\frac{E_{i+1}'}{E_{i+1}}].$$

Now, for any $[T;\frac{E'}{E}]\in \mathcal{F}_1(\Sigma_1)$, we have 
$[T;\frac{E'}{E}]=-[T(\frac{E'}{E});\frac{E}{E'}].$
\end{proof}

We shall get rid of the elementary surgeries of genus 1 with the help of the following two lemmas.
\begin{lemma} \label{lemmasub}
 Let $E$ be a framed standard torus. Let $E'$ be a framed $d$-torus. Assume $\ell(E')=d\gamma$ in $H_1(E';\mathbb{Z})$ 
for a curve $\gamma$ in $E'$. Embed two disjoint copies $E_1$ and $E_2$ of $E$ in $Int(E)$ so that $\ell(E_1)=\ell(E_2)=\ell(E)$ 
in $H_1(E\setminus Int(E_1\cup E_2);\mathbb{Z})$. Let $E_1'$ and $E_2'$ be two copies of $E'$.
Set $A=E(\frac{E_1'}{E_1},\frac{E_2'}{E_2})$. Then there is a $\mathbb{Q}$HS $M$ 
such that $A$ can be obtained from $E'\sharp M$ by a finite sequence of borromean surgeries.
\end{lemma}
\begin{proof}
For $i=1,2$, denote by $\gamma_i$ the copy of $\gamma$ in $E_i'$, so that $\ell(E_i')=d\gamma_i$ in $H_1(E_i';\mathbb{Z})$. 
Note that, in $H_1(A)$, $\ell(A)=\ell(E_1')=\ell(E_2')$, and $m(A)=m(E_1')+m(E_2')$. We have~:
\begin{eqnarray*}
 H_1(A;\mathbb{Z}) &=& <m(E_1'),m(E_2'),\gamma_1,\gamma_2|d\,m(E_1')=0,d\,m(E_2')=0,d\gamma_1=d\gamma_2> \\
  &=& <m(E_1'),m(A),\gamma_1,\gamma_2-\gamma_1|d\,m(E_1')=0,d\,m(A)=0,d(\gamma_2-\gamma_1)=0> \\
  &=& \mathbb{Z}_d\, m(A) \oplus \mathbb{Z}\, \gamma_1 \oplus \mathbb{Z}_d\, m(E_1') \oplus \mathbb{Z}_d\, (\gamma_2-\gamma_1)
\end{eqnarray*}
Note that $\ell(A)=d\gamma_1$. 
Consider a tunnel $C$ around $\gamma_1$. Set $B=\overline{A\setminus C}$. There is a surface $S\subset B$ such that 
$\partial S\subset \partial B$ is homologous to $\ell(A) -d \ell+km$ in $\partial B$, where $m$ is a meridian of $\gamma_1$, $\ell$ is a 
longitude of $\gamma_1$, and $k$ is an integer. Consider simple closed curves $\sigma_1$, $\sigma_2$, $\mu_1$ and $\mu_2$ in $\partial B$
such that $\sigma_1=m-dm(A)$, $\sigma_2=\ell(A) -d \ell+k m$, $\mu_1=-\ell+km(A)$ and $\mu_2=-m(A)$ in $H_1(\partial B)$.
The curves $\sigma_1$ and $\sigma_2$ bound embedded surfaces in $B$, and 
$(\sigma_1,\mu_1,\sigma_2,\mu_2)$ is a symplectic basis of $H_1(\partial B;\mathbb{Z})$. 
Thus $B$ is a genus $2$ $\mathbb{Q}$HH with $|\mathcal{L}_B^T/\mathcal{L}_B^\mathbb{Z}|=0$.
By Lemma \ref{corcaspart}, there are a standard genus 2 handlebody $H_2$ and a $\mathbb{Q}$HS $M$ 
such that $B$ is obtained from $H_2\sharp M$ by a finite sequence of borromean surgeries.

Now consider the $d$-torus $E'$. It is homeomorphic to $E(\frac{E_1'}{E_1})$. Consider a tunnel $C'$ around $\gamma$ in $E'$.
We can choose a meridian $m'$ and a longitude $\ell'$ of $\gamma$ in such a way that 
there are curves $\sigma'_1$ and $\sigma'_2$ on the boundary of $B'=\overline{E'\setminus C'}$ which bound surfaces in $B'$
and which are respectively homologous to $m'-d\alpha$ and $\beta -d \ell'+km'$ in $H_1(\partial B';\mathbb{Z})$.
Thus the LP-identification $\partial E'\cong \partial A$ extends to an LP-identification 
$\partial B'\cong\partial B \cong\partial H_2$. Since $H_1(B';\mathbb{Z})=\mathbb{Z}(\ell'-k\alpha)\oplus\mathbb{Z}m'$, 
$B'$ and $H_2$ are two $\mathbb{Z}$HH whose boundaries are LP-identified. By \cite[Lemma 4.11]{AL}, 
$H_2$ can be obtained from $B'$ by a finite sequence of borromean surgeries. Thus $H_2\sharp M$, and $B$, can be 
obtained from $B'\sharp M$ by a finite sequence of borromean surgeries. Gluing back the cylinders, we see 
that $A$ can be obtained from $E'\sharp M$ by a finite sequence of borromean surgeries. 
\end{proof}
\begin{lemma} \label{lemmatype2}
 The quotient $\mathcal{G}_1(\Sigma_1)$ is generated by the $[T;\frac{E'}{E}]$, where $(\frac{E'}{E})$
is an elementary surgery of genus 0 (connected sum) or 3 (borromean surgery).
\end{lemma}
\begin{proof}
Consider a framed rational homology torus $T$ and an elementary surgery 
$\frac{E'}{E}$ of genus 1 in $T$, \textit{i.e.} $E$ is an embedded standard torus, and $E'$ is a $d$-torus. 
Fix a longitude $\ell(E')$ such that $\ell(E')=d\gamma$ in $H_1(E';\mathbb{Z})$ for a curve $\gamma$ in $E'$. 
Choose the longitude $\ell(E)$ which is identified with $\ell(E')$ by the LP-identification $\partial E\cong\partial E'$. 

Consider the copies $E_1$ and $E_2$ of $E$ in $Int(E)$, the copies $E_1'$ and $E_2'$ of $E'$, 
the rational homology torus $A$, and the $\mathbb{Q}$HS $M$, defined in Lemma \ref{lemmasub}. 
Set $T'=T(\frac{E_1'}{E_1},\frac{E_2'}{E_2})\cong T(\frac{A}{E})$. 
Write $A=E'(\frac{B(M)}{B^3})(\frac{B_1'}{B_1})(\frac{B_2'}{B_2})\dots(\frac{B_k'}{B_k})$, where $B(M)$ 
is the rational homology ball obtained by removing a ball $B^3$ from $M$, and the $(\frac{B_i'}{B_i})$ are borromean surgeries. 
On the one hand, we have $[T;\frac{E_1'}{E_1},\frac{E_2'}{E_2}]=2[T;\frac{E'}{E}]-T+T'$, thus:
$$T-T'=2[T;\frac{E'}{E}]\quad mod\ \mathcal{F}_2(\Sigma_1),$$ and, on the other hand:
$$T-T'= [T;\frac{A}{E}]=[T;\frac{E'}{E}]+[T(\frac{E'}{E});\frac{B(M)}{B^3}]
+\sum_{i=1}^k [T(\frac{E'}{E})(\frac{B(M)}{B^3})(\frac{B_1'}{B_1})..(\frac{B_{i-1}'}{B_{i-1}});\frac{B_i'}{B_i}]$$

Thus: 
$$\hspace{-1cm} [T;\frac{E'}{E}]=[T(\frac{E'}{E});\frac{B(M)}{B^3}]+\sum_{i=1}^k [T(\frac{E'}{E})(\frac{B(M)}{B^3})
(\frac{B_1'}{B_1})..(\frac{B_{i-1}'}{B_{i-1}});\frac{B_i'}{B_i}]\quad mod\ \mathcal{F}_2(\Sigma_1).$$
\end{proof}

We shall now restrict the set of generators $[T;\frac{E'}{E}]$ where $(\frac{E'}{E})$ is an elementary surgery of genus 0.
\begin{lemma} \label{lemmatype1}
Let $T$ be a framed rational homology torus and let $B$ be a rational homology ball.
Then $[T;\frac{B}{B^3}]\in\bigoplus_{p\ prime}\mathbb{Q} [T_0;\frac{B_p}{B^3}]\subset\mathcal{G}_1(\Sigma_1)$.
\end{lemma}
This result follows from the next two sublemmas.
\begin{sublemma}
 Let $T$ be a framed rational homology torus and let $B$ be a rational homology ball. Then, in $\mathcal{G}_1(\Sigma_1)$,
 $[T;\frac{B}{B^3}]$ is a linear combination of the $[T;\frac{B_p}{B^3}]$.
\end{sublemma}
\begin{proof}
 Set $M=B\cup_{\partial B =-\partial B^3} B^3$. We have $T(\frac{B}{B^3})=T\sharp M$ and $T(\frac{B_p}{B^3})=T\sharp M_p$.
Now use that $(M_p-S^3)_{p\ prime}$ generates $\mathcal{G}_1$ (see Subsection \ref{subsecgene}).
\end{proof}
\begin{sublemma}
 For any framed rational homology torus $T$, and any rational homology ball $B$, 
$$[T;\frac{B}{B^3}]=[T_0;\frac{B}{B^3}]\quad mod\ \mathcal{F}_2(\Sigma_1).$$
\end{sublemma}
\begin{proof}
Define $T_0'$ as $T_0$ minus a regular open neighborhood of its boundary. We can suppose that 
$T_0'$ and $B^3$ are disjoint in $T_0$. We have $T\cong T_0(\frac{T}{T_0'})$, and:
$$[T_0;\frac{T}{T_0'},\frac{B}{B^3}]=[T_0;\frac{B}{B^3}]-[T;\frac{B}{B^3}].$$
\end{proof}

\proofof{Proposition \ref{propinvtori}} 
By Lemmas \ref{lemmatype2} and \ref{lemmatype1}, $\mathcal{G}_1(\Sigma_1)$ is generated by the $[T_0;\frac{B_p}{B^3}]$ and 
the $[T;\frac{A'}{A}]$ where $(\frac{A'}{A})$ is a borromean surgery. Consider $\mu\in(\mathcal{G}_1(\Sigma_1))^*$. 
For all prime integer $p$, set $c_p=\mu([T_0;\frac{B_p}{B^3}])$. Set $\tilde{\mu}=\mu-\sum_{p\ prime}c_p\mu_p$.
The invariant $\tilde{\mu}$ is determined by its values on the terms $[T;\frac{A'}{A}]$, where $(\frac{A'}{A})$ is a 
borromean surgery. Let $\Gamma$ denote the Y-graph associated with the borromean surgery $(\frac{A'}{A})$.
By Lemma \ref{lemmaF1}, if $T$ is fixed, $\tilde{\mu}([T;\frac{A'}{A}])$ only depends on the rational homology classes of 
the three leaves of $\Gamma$, and this dependance is trilinear and alternating.
Since $H_1(T;\mathbb{Q})\cong\mathbb{Q}$, we have $\tilde{\mu}=0$. Hence $\mu=\sum_{p\ prime}c_p\mu_p$. 
This implies that $\mathcal{G}_1(\Sigma_1)$ is generated by the $[T_0;\frac{B_p}{B^3}]$. Conclude with Corollary \ref{cordirect}.
\fin

\begin{corollary} \label{cortori}
 If $\mu$ is a degree 1 invariant of framed rational homology tori, such that $\mu(T_0)=0$ and $\mu(T_0\sharp M_p)=0$
for any prime $p$, then $\mu=0$.
\end{corollary}

      \subsection{The quotients $\frac{\mathcal{I}_n^c}{\mathcal{I}_{n-1}^c}$}

The main point of this subsection will be the proof of the next proposition. We will end the subsection by showing that this result 
implies Proposition \ref{propdual}.

\begin{proposition} \label{propgn}
 If $\lambda$ is an additive invariant of degree $n>1$, then $\lambda_{|\mathcal{F}_n}$ is determined by 
$\lambda(\Phi(\mathcal{A}_{\frac{n}{2}}^c))$. In particular, if $n$ is odd, $\lambda_{|\mathcal{F}_n}=0$.
\end{proposition}
Recall the map $\Phi : \mathcal{A}_n \to \mathcal{G}_{2n}^\mathbb{Z}$ has been defined in Lemma \ref{lemmaphi}. 
Since we have a canonical map $\mathcal{G}_{2n}^\mathbb{Z}\to\mathcal{G}_{2n}$ (we will see later that it is an embedding), 
$\lambda(\Phi(\mathcal{A}_n^c))$ is well defined.

We will often use the following easy formula.
\begin{lemma}
For any $[M;(\frac{A_i'}{A_i})_{1\leq i\leq n}]\in\mathcal{F}_n$,
$$[M;(\frac{A_i'}{A_i})_{1\leq i\leq n}]=[M;(\frac{A_i'}{A_i})_{2\leq i\leq n}]-[M(\frac{A_1'}{A_1});(\frac{A_i'}{A_i})_{2\leq i\leq n}].$$
\end{lemma}

\begin{lemma} \label{lemmared}
The space $\mathcal{G}_n$ is generated by the $[M;(\frac{E_i'}{E_i})_{1\leq i\leq n}]$, where the $(\frac{E_i'}{E_i})$ 
are elementary surgeries.
\end{lemma}
To see this, just adapt the proof of Lemma \ref{lemmaredtori}.

\begin{lemma} \label{lemmat1}
 Let $\lambda$ be an additive invariant of degree $n>1$. Consider $[M;(\frac{E_i'}{E_i})_{1\leq i\leq n}] \in\mathcal{F}_n$.
If at least one of the surgeries $(\frac{E_i'}{E_i})$ is an elementary surgery of genus 0 (connected sum), 
then  $\lambda([M;(\frac{E_i'}{E_i})_{1\leq i\leq n}])=0$.
\end{lemma}
\begin{proof}
 Assume $\frac{E_1'}{E_1}$ is a connected sum, {\em i.e.} $E_1$ is a ball $B^3$, and $E_1'$ is a rational homology ball.
Define a $\mathbb{Q}$HS $M_1$ by gluing $E_1'$ and a ball $B^3$ along their boundaries.
\begin{eqnarray*}
 \lambda([M;(\frac{E_i'}{E_i})_{1\leq i\leq n}])&=&\lambda([M;(\frac{E_i'}{E_i})_{2\leq i\leq n}])
-\lambda([M\sharp M_1;(\frac{E_i'}{E_i})_{2\leq i\leq n}])\\ &=&-\sum_{I\subset \{2,\dots,n\}} (-1)^{|I|}\lambda(M_1)\\ &=& 0
\end{eqnarray*}
\end{proof}

\begin{lemma} \label{lemmat2}
 Consider a $\mathbb{Q}$HS $M$, and disjoint LP-surgeries $(\frac{T_d}{T_0}),(\frac{A_i'}{A_i})_{1\leq i\leq n-1}$ in $M$, 
where $T_0$ is a standard torus, and $T_d$ is a $d$-torus.
If $\lambda$ is an additive invariant of degree $n>1$, then $\lambda([M;\frac{T_d}{T_0},(\frac{A_i'}{A_i})_{1\leq i\leq n-1}])=0$.
\end{lemma}
\begin{proof}
 Fix $M$, the embedding of $T_0$, and the surgeries $(\frac{A_i'}{A_i})_{1\leq i\leq n-1}$. Fix a longitude $\ell(T_0)$ of $T_0$. 
For any framed rational homology torus $T$, set $\bar{\lambda}(T)=\lambda([M;\frac{T}{T_0},(\frac{A_i'}{A_i})_{1\leq i\leq n-1}])$.
Then $\bar{\lambda}$ is a degree 1 invariant of framed rational homology tori:
$$\bar{\lambda}([T;\frac{B_1'}{B_1},\frac{B_2'}{B_2}]) 
= \lambda\lbp -[M(\frac{T}{T_0});\frac{B_1'}{B_1},\frac{B_2'}{B_2},(\frac{A_i'}{A_i})_{1\leq i\leq n-1}] \rbp 
= 0.$$
We have $\bar{\lambda}(T_0)=\lambda(0)=0$, and:
$$\bar{\lambda}(T_0\sharp M_p) = \lambda\lbp [M;\frac{B_p}{B^3},(\frac{A_i'}{A_i})_{1\leq i\leq n-1}]\rbp =0,$$
since $\lambda$ is additive, and $n-1>0$.
By Corollary \ref{cortori}, $\bar{\lambda}=0$.
\end{proof}

\proofof{Proposition \ref{propgn}} 
By Lemmas \ref{lemmat1} and \ref{lemmat2}, an additive invariant $\lambda$ 
of degree $n>1$ is determined on $\mathcal{F}_n$ by its values on the 
$[M;(\frac{B_i'}{B_i})_{1\leq i\leq n}]$, for all $\mathbb{Q}$HS's $M$ and all sets of $n$ disjoint borromean surgeries 
$(\frac{B_i'}{B_i})_{1\leq i\leq n}$ in $M$. Hence, by Corollary \ref{corJacobi}, $\lambda$ is determined on $\mathcal{F}_n$ 
by the $\lambda([M;\Gamma])$ for all $\mathbb{Q}$HS $M$ and all Jacobi diagram $\Gamma$ of degree $\frac{n}{2}$. 

We can write $M=M\sharp S^3$ and suppose $\Gamma$ is embedded in $S^3$. Hence for an additive invariant $\lambda$ of degree $n$, 
we have $\lambda([M;\Gamma])=\lambda([S^3;\Gamma])$. 

If the Jacobi diagram $\Gamma$ is not connected, we can assume that $\Gamma$ is made of two components $\Gamma_1$ and $\Gamma_2$ 
that are embedded in disjoint balls in $S^3$. Noting that $(S^3,\Gamma)=(S^3,\Gamma_1)\sharp(S^3,\Gamma_2)$, it is easy 
to see that any additive invariant vanishes on $[S^3;\Gamma]$ in this case.
\fin

Proposition \ref{propdual} follows from Proposition \ref{propgn} in the case of odd degrees. 
For even degrees, it is a consequence of the following lemma.
\begin{lemma} \label{lemmainvadd}
Let $n>1$ be an even integer. 
Let $(\Gamma_{n,i})_{i\in C_n}$ be a basis of diagrams of the finite dimensional vector space $\mathcal{A}_{\frac{n}{2}}^c$. 
Let $(\Gamma_{n,i}^*)_{i\in C_n}$ be the dual basis of $(\mathcal{A}_\frac{n}{2}^c)^*$. 
Let $Z_{\frac{n}{2}}$ denote the degree $\frac{n}{2}$ part of the KKT invariant. 
Let $p^c: \mathcal{A}_\frac{n}{2}\to\mathcal{A}_\frac{n}{2}^c$ be the projection that maps any non 
connected diagram to 0 and which restricts to the identity on $\mathcal{A}_\frac{n}{2}^c$. 
For $i\in C_n$, set $\lambda_{n,i}=\Gamma_{n,i}^*\circ p^c \circ Z_\frac{n}{2}$. 
Then $(\lambda_{n,i})_{i\in C_n}$ is a basis of $\frac{\mathcal{I}_n^c}{\mathcal{I}_{n-1}^c}$.
\end{lemma}
\begin{proof}
By \cite[Theorem 1]{KT}, $p^c\circ Z_\frac{n}{2}$ is an additive invariant of $\mathbb{Q}$HS's, thus the $\lambda_{n,i}$ are additive. 
By \cite[Theorem 2.4]{Les} and \cite[Proposition 4.1]{AL}, $Z_\frac{n}{2}$ is a finite type invariant of degree n and satisfies 
$Z_\frac{n}{2}([S^3;\Gamma_{n,i}])=\Gamma_{n,i}\in\mathcal{A}_{\frac{n}{2}}$. Hence $\lambda_{n,i}\in\mathcal{I}_n^c$, 
and $\lambda_{n,i}([S^3;\Gamma_{n,j}])=\delta_{i,j}$. Consider $\lambda\in\mathcal{I}_n^c$. By Proposition \ref{propgn}, 
$\lambda=\sum_{i\in C_n} \lambda([S^3;\Gamma_{n,i}])\lambda_{n,i}$ in $\frac{\mathcal{I}_n^c}{\mathcal{I}_{n-1}^c}$. 
Hence $(\lambda_{n,i})_{i\in C_n}$ is a basis of $\frac{\mathcal{I}_n^c}{\mathcal{I}_{n-1}^c}$.
\end{proof}

    \section{The graded algebras $\mathcal{G}$ and $\mathcal{H}$} \label{secfil}

      \subsection{The products in $\mathcal{G}$ and $\mathcal{H}$}

Extend the connected sum to $\mathcal{F}_0$ by bilinearity: 
$$(\sum_{i\in I}a_iM_i) \sharp (\sum_{j\in J}b_jN_j) = \sum_{i\in I}\sum_{j\in J} a_ib_j(M_i\sharp N_j),$$
for any finite sets $I$ and $J$, any rational numbers $a_i$, $b_j$, and any $\mathbb{Q}$HS's $M_i$, $N_j$.
\begin{lemma}
$\mathcal{F}_n \sharp \mathcal{F}_m \subset \mathcal{F}_{n+m}$
\end{lemma}
\begin{proof}
Just check that 
$[M;(\frac{B_i}{A_i})_{1\leq i \leq n}]\sharp [M';(\frac{B'_i}{A'_i})_{1\leq i \leq m}]=
[M\sharp M';(\frac{B_i}{A_i})_{1\leq i \leq n},(\frac{B'_i}{A'_i})_{1\leq i \leq m}]$.
\end{proof}
 
Thus the connected sum defines a product $\sharp: \mathcal{G}_n\times\mathcal{G}_m\to \mathcal{G}_{n+m}$ which
induces a graded algebra structure on $\mathcal{G}$.

Given two finite type invariants $\lambda$ and $\mu$, note that the product $\lambda\mu$ satisfies:
$$\lambda\mu(\sum_{i\in I}a_iM_i)=\sum_{i\in I}a_i\lambda(M_i)\mu(M_i),$$ for 
any finite set $I$, rational numbers $a_i$, and $\mathbb{Q}$HS's $M_i$.
\begin{lemma}
 If $\lambda\in\mathcal{I}_k$ and $\mu\in\mathcal{I}_\ell$, then $\lambda\mu\in\mathcal{I}_{k+\ell}$.
\end{lemma}
\begin{proof}
Consider $[M;(\frac{B_i}{A_i})_{i\in I}]$ with $|I|=k+\ell+1$. We have the following equality.
\begin{eqnarray} \label{eq}
 \lambda\mu([M;(\frac{B_i}{A_i})_{i\in I}])=\sum_{J\subset I}
 \lambda([M;(\frac{B_i}{A_i})_{i\in J}])\mu([M((\frac{B_i}{A_i})_{i\in J});(\frac{B_i}{A_i})_{i\in I\setminus J}])
\end{eqnarray}
Indeed, the right hand side is equal to: 
\begin{eqnarray*}
&&\sum_{J\subset I} (-1)^{|J|} \lbp \sum_{K\subset J} (-1)^{|K|} \lambda(M((\frac{B_i}{A_i})_{i\in K}))\rbp 
 \lbp \sum_{L\supset J} (-1)^{|L|} \mu(M((\frac{B_i}{A_i})_{i\in L}))\rbp \\
&=&\sum_{L\subset I}\sum_{K\subset L} (-1)^{|K|+|L|} \lambda(M((\frac{B_i}{A_i})_{i\in K})) \mu(M((\frac{B_i}{A_i})_{i\in L}))
 \lbp \sum_{K\subset J\subset L} (-1)^{|J|} \rbp.
\end{eqnarray*}
Since $\sum_{K\subset J\subset L} (-1)^{|J|}=\left\lbrace 
\begin{array}{l l} 0 & if\ K\subsetneq L \\ (-1)^{|K|} & if\ K=L \end{array}\right.$, we get (\ref{eq}).\\ [4pt]
In (\ref{eq}), we have, if $|J|>k$, $$\lambda([M;(\frac{B_i}{A_i})_{i\in J}])=0,$$ and, if $|J|\leq k$, then $|I\setminus J|>\ell$ and 
$$\mu([M((\frac{B_i}{A_i})_{i\in J});(\frac{B_i}{A_i})_{i\in I\setminus J}])=0.$$ Thus $$\lambda\mu([M;(\frac{B_i}{A_i})_{i\in I}])=0.$$
\end{proof}
Thus the product of finite type invariants induces a graded algebra structure on $\mathcal{H}$.

      \subsection{Dual systems in $\mathcal{G}$ and $\mathcal{H}$} \label{subsecdual}

For an even integer $n>1$, consider the basis $(\Gamma_{n,i})_{i\in C_n}$ of $\mathcal{A}_\frac{n}{2}^c$ and 
the associated invariants $\lambda_{n,i}$ defined in Lemma \ref{lemmainvadd}. For $n>1$ odd, set $C_n=\emptyset$. 
For $n=1$, let $C_1$ denote the set of all prime integers, and for any $p$ prime, set $\lambda_{1,p}=\nu_p$ 
and $\Gamma_{1,p}=\bullet_p\in\mathcal{A}_1^{aug}$. 
Note that adding to $\lambda_{n,i}$ a weighted sum of the $\lambda_{k,i}$, $0<k<n$, $i\in C_k$, does not change the values of 
$\lambda_{n,i}$ on $\mathcal{F}_n$. Thus we can (and we do) choose the basis $(\lambda_{n,i})_{i\in C_n}$ so that
$\lambda_{n,i}([S^3;\Gamma_{k,j}])=\delta_{nk}\delta_{ij}$ for all positive integers $n$ and $k$, all $i\in C_n$, all $j\in C_k$.

For a multi-index $\me=(\varepsilon_t)_{1\leq t\leq \ell}$, set $\ell(\me)=\ell$. For $n>0$, fix a total order on $C_n$. 
Let $\preccurlyeq$ denote the lexicographic order induced on $\bigcup_{n\in\mathbb{N}\setminus\{0\}} (\{n\}\times C_n)$. 
For $n>0$, let $\mathcal{T}_n^\pi$ denote the set of all triples $(\mk,\mi,\me)$ such that $\ell(\mk)=\ell(\mi)=\ell(\me)$, 
$\mk=(k_t)_{1\leq t\leq \ell(\mk)}$, $k_t\in\mathbb{N}$ and $0<k_t<n$ for all $t$, $\mi=(i_t)_{1\leq t\leq \ell(\mk)}$ 
with $i_t\in C_{k_t}$ for all $t$, $(k_1,i_1)\prec (k_2,i_2)\prec \dots\prec (k_{\ell(\mk)},i_{\ell(\mk)})$, 
$\me=(\varepsilon_t)_{1\leq t\leq \ell(\me)}$ with $\varepsilon_t\in\mathbb{N}\setminus \{0\}$ 
for all $t$, and $\sum_{1\leq t\leq \ell(\mk)}\varepsilon_t k_t=n$. Define a family $(\lambda_{n,\iota})_{\iota\in \mathcal{T}_n^\pi}$ 
of invariants of degree $n$ by $\lambda_{n,\iota}=\prod_{1\leq t\leq \ell(\mk)}\lambda_{k_t,i_t}^{\varepsilon_t}$ if $\iota=(\mk,\mi,\me)$.
Set $\mathcal{T}_n=C_n\sqcup\mathcal{T}_n^\pi$. We will see in Subsection \ref{subsecdec} that the family 
$(\lambda_{k,i})_{\begin{subarray}{l} 0<k\leq n \\ i\in \mathcal{T}_k \end{subarray}}$ is a basis  of $\frac{\mathcal{I}_n}{\mathcal{I}_0}$. 
The main goal of this subsection is to construct a family $(G_{k,i}^{(n)})_{\begin{subarray}{l} 0<k\leq n \\ i\in \mathcal{T}_k \end{subarray}}$ 
of $\frac{\mathcal{F}_1}{\mathcal{F}_{n+1}}$, dual to $(\lambda_{k,i})_{\begin{subarray}{l} 0<k\leq n \\ i\in \mathcal{T}_k \end{subarray}}$.

\begin{definition}
 $G\in\mathcal{F}_0$ is said to be $\emph{multiplicative}$ if $\lambda\mu(G)=\lambda(G)\mu(G)$ for all finite type invariants 
$\lambda$ and $\mu$ such that $\lambda(S^3)=0$ and $\mu(S^3)=0$.
\end{definition}
For any $p$ prime, set $G_{1,p}^{(1)}=M_p-S^3$. Note that the $G_{1,p}^{(1)}$ are multiplicative. 
Fix $n>1$. If $n$ is even, set $G_{n,i}^{(n)}=[S^3;\Gamma_{n,i}]$ for $i\in C_n$.
Since $[S^3;\Gamma_{n,i}]=S^3(\Gamma_{n,i})-S^3$, $G_{n,i}^{(n)}$ is multiplicative for all $i\in C_n$.
For $\iota=(\mk,\mi,\me)\in\mathcal{T}_n^\pi$, set 
$\tilde{G}_{n,\iota}^{(n)}=\cs{1\leq t\leq \ell(\mk)} (G_{k_t,i_t}^{(k_t)})^{\sharp\,\varepsilon_t}$. 

\begin{lemma} \label{lemmaprod}
 Consider positive integers $p$ and $q$, additive invariants $\lambda_1$, $\dots$, $\lambda_p$, and elements 
$[M_1;(\frac{B_u}{A_u})_{u\in U_1}]$, $\dots$, $[M_q;(\frac{B_u}{A_u})_{u\in U_q}]$ of 
$\mathcal{F}_0$, for non empty sets $U_j$. Then:
$$(\prod_{i=1}^p \lambda_i)(\cs{j=1}^q [M_j;(\frac{B_u}{A_u})_{u\in U_j}])=
\sum_{j\in E_{pq}} \prod_{\ell=1}^q (\prod_{i\in j^{-1}(\{\ell\})} \lambda_i)([M_\ell;(\frac{B_u}{A_u})_{u\in U_\ell}]),$$
where $E_{pq}$ is the set of all surjective maps $j: \{1,..,p\}\to \{1,..,q\}$.

In particular, if $p<q$, 
$$(\prod_{i=1}^p \lambda_i)(\cs{j=1}^q [M_j;(\frac{B_u}{A_u})_{u\in U_j}])=0,$$
and, if $p=q$,
$$(\prod_{i=1}^p \lambda_i)(\cs{j=1}^p [M_j;(\frac{B_u}{A_u})_{u\in U_j}])=
\sum_{\sigma\in\mathcal{S}_p}\prod_{\ell=1}^p \lambda_{\sigma(\ell)}([M_\ell;(\frac{B_u}{A_u})_{u\in U_\ell}]),$$
where $\mathcal{S}_p$ is the set of permutations of $\{1,\dots,p\}$.
\end{lemma}
\begin{proof} \ \\
 $\displaystyle (\prod_{i=1}^p \lambda_i)(\cs{j=1}^q [M_j;(\frac{B_u}{A_u})_{u\in U_j}]) $ 
\vspace{-3ex} \begin{eqnarray*}
  &=& \sum_{V_1\subset U_1}\dots\sum_{V_q\subset U_q}(-1)^{\sum_{\ell=1}^q |V_\ell|} 
    \prod_{i=1}^p\lambda_i\lbp\cs{j=1}^q M_j((\frac{B_u}{A_u})_{u\in V_j})\rbp \\
  &=& \sum_{V_1\subset U_1}\dots\sum_{V_q\subset U_q}(-1)^{\sum_{\ell=1}^q |V_\ell|} 
    \prod_{i=1}^p\sum_{j=1}^q \lambda_i\lbp M_j((\frac{B_u}{A_u})_{u\in V_j})\rbp \\
  &=& \sum_{j:\{1,..,p\}\to\{1,..,q\}} \sum_{V_1\subset U_1}\dots\sum_{V_q\subset U_q}(-1)^{\sum_{\ell=1}^q |V_\ell|} 
    \prod_{i=1}^p\lambda_i\lbp M_{j(i)}((\frac{B_u}{A_u})_{u\in V_{j(i)}})\rbp \\
  &=& \sum_{j:\{1,..,p\}\to\{1,..,q\}} \prod_{\ell=1}^q \lhp\sum_{V_\ell\subset U_\ell}(-1)^{|V_\ell|}
    \prod_{i\in j^{-1}(\{\ell\})}\lambda_i\lbp M_\ell((\frac{B_u}{A_u})_{u\in V_\ell})\rbp\rhp\\
  &=& \sum_{j\in E_{pq}} \prod_{\ell=1}^q (\prod_{i\in j^{-1}(\{\ell\})} \lambda_i)([M_\ell;(\frac{B_u}{A_u})_{u\in U_\ell}])
\end{eqnarray*}
\end{proof}

\begin{lemma} \label{lemmad1}
 Let $n$ and $k$ be positive integers. For $\iota=(\mk,\mi,\me)\in\mathcal{T}_k^\pi$, set 
$\mathcal{T}_n(\iota)=\{(\mk,\mi,\meta)\in\mathcal{T}_n^\pi\,|\,\forall t,\,\eta_t\geq\varepsilon_t\}$. 
For $\kappa\in\mathcal{T}_n^\pi$, we have $\lambda_{n,\kappa}(\tilde{G}_{k,\iota}^{(k)})\neq 0$ 
if and only if $\kappa\in\mathcal{T}_n(\iota)$.
\end{lemma}
Note that the set $\mathcal{T}_n(\iota)$ is finite.
\begin{proof}
 Set $\kappa=(\ml,\mj,\meta)$. We have $\lambda_{n,\kappa}=\prod_{1\leq s\leq \ell(\ml)} \lambda_{\ell_s,j_s}^{\eta_s}$ 
and $\tilde{G}_{k,\iota}^{(k)}=\cs{1\leq t\leq \ell(\mk)} (G_{k_t,i_t}^{(k_t)})^{\sharp \varepsilon_t}$. By Lemma \ref{lemmaprod}:
$$\lambda_{n,\kappa}(\tilde{G}_{k,\iota}^{(k)})=
\sum_{\xi\in E_{\meta\me}} \prod_{(t,u)\in\Theta(\me)}
\lbp \prod_{(s,v)\in\xi^{-1}(\{(t,u)\})} \lambda_{\ell_s,j_s}\rbp (G_{k_t,i_t}^{(k_t)}),$$
where $\Theta(\me)=\{(t,u)|\,1\leq t\leq \ell(\me);\forall\,t,1\leq u\leq \varepsilon_t\}$ 
and $E_{\meta\me}$ is the set of all surjective maps $\xi:\Theta(\meta)\twoheadrightarrow\Theta(\me)$.
Since the $G_{k_t,i_t}^{(k_t)}$ are multiplicative, we get:
$$\lambda_{n,\kappa}\lbp\tilde{G}_{k,\iota}^{(k)}\rbp=
\sum_{\xi\in E_{\meta\me}} \prod_{(t,u)\in\Theta(\me)}
\prod_{(s,v)\in\xi^{-1}(\{(t,u)\})}\lbp  \lambda_{\ell_s,j_s} (G_{k_t,i_t}^{(k_t)})\rbp.$$
Recall $\lambda_{\ell_s,j_s} (G_{k_t,i_t}^{(k_t)})=\delta_{\ell_sk_t}\delta_{j_si_t}$. 
Hence $\lambda_{n,\kappa}\lbp\tilde{G}_{k,\iota}^{(k)}\rbp\neq 0$
if and only if $\ml=\mk$, $\mj=\mi$ and $\eta_t\geq\varepsilon_t$ for all $t$.
\end{proof}

For $n>1$ and $\iota\in\mathcal{T}_n^\pi$, set 
$\displaystyle G_{n,\iota}^{(n)}=\frac{1}{\lambda_{n,\iota}(\tilde{G}_{n,\iota}^{(n)})}\tilde{G}_{n,\iota}^{(n)}$, 
so that $\lambda_{n,\iota}(G_{n,\iota}^{(n)})=1$. 
Note that, for all $n$ and all $i\in \mathcal{T}_n$, $G_{n,i}^{(n)}\in\mathcal{F}_n$.

Let $n$ and $k$ be positive integers. For $\iota\in C_k$, set 
$\mathcal{T}_n(\iota)=\{(\mk,\mi,\meta)\in\mathcal{T}_n^\pi\,|\,\mk=(k),\mi=(\iota)\}$. 
The following result is an easy generalisation of Lemma \ref{lemmad1}.
\begin{lemma} \label{lemmad1'}
 Let $n$ and $k$ be positive integers. For $\kappa\in \mathcal{T}_n$ and $\iota\in \mathcal{T}_k$, we have 
$\lambda_{n,\kappa}(G_{k,\iota}^{(k)})\neq 0$ if and only if $\kappa\in\mathcal{T}_n(\iota)$.
\end{lemma}

\begin{corollary} \label{cord2}
 For $n>0$, $i\in \mathcal{T}_n$, $j\in \mathcal{T}_n$, we have $\lambda_{n,i}(G^{(n)}_{n,j})=\delta_{ij}$.
\end{corollary}

For $n>1$, define $G_{k,i}^{(n)}\in\mathcal{F}_k$ for $0<k<n$ and $i\in \mathcal{T}_k$, by induction on $n$, by: 
$$G_{k,i}^{(n)}=G_{k,i}^{(n-1)}-\sum_{\iota\in \mathcal{T}_n(i)} \lambda_{n,\iota}(G_{k,i}^{(n-1)}) G_{n,\iota}^{(n)}.$$
Note that $G_{k,i}^{(n)}=G_{k,i}^{(m)}$ in $\mathcal{G}_m$ if $m\leq n$.

\begin{lemma} \label{lemmad3}
 Let $n$ be a positive integer. The family $(G^{(n)}_{k,i})_{\begin{subarray}{l} 0<k\leq n \\ i\in \mathcal{T}_k \end{subarray}}$ of 
$\frac{\mathcal{F}_1}{\mathcal{F}_{n+1}}$ is dual to the family 
$(\lambda_{k,i})_{\begin{subarray}{l} 0<k\leq n \\ i\in \mathcal{T}_k \end{subarray}}$ of $\frac{\mathcal{I}_n}{\mathcal{I}_0}$.
\end{lemma}
\begin{proof}
We proceed by induction on $n$. The result is clear for $n=1$. Fix $n>1$. We shall prove that 
$\lambda_{\ell,j}(G_{k,i}^{(n)})=\delta_{\ell k}\delta_{ji}$ for all $0<\ell\leq n$, $j\in \mathcal{T}_\ell$, $0<k\leq n$, $i\in \mathcal{T}_k$. 
If $\ell=n$ and $k=n$, it is given by Corollary \ref{cord2}. If $\ell<n$ and $k=n$, it is clear since $G_{n,i}^{(n)}\in\mathcal{F}_n$. 
If $\ell<n$ and $k<n$, it follows from the induction hypothesis. It remains to show that $\lambda_{n,j}(G_{k,i}^{(n)})=0$ if $k<n$.
It is immediate if $j\in\mathcal{T}_n(i)$. Consider $j\in \mathcal{T}_n\setminus\mathcal{T}_n(i)$. We have: 
$$G_{k,i}^{(n)}=G_{k,i}^{(k)}-\sum_{k<m\leq n} \sum_{\iota\in\mathcal{T}_m(i)}\lambda_{m,\iota}(G_{k,i}^{(m-1)}) G_{m,\iota}^{(m)}.$$
By Lemma \ref{lemmad1'}, for $k\leq m\leq n$ and $\iota\in\mathcal{T}_m(i)$, $\lambda_{n,j}(G_{m,\iota}^{(m)})\neq 0$ 
if and only if $j\in\mathcal{T}_n(\iota)$, and this implies $j\in\mathcal{T}_n(i)$. Hence, for $j\notin\mathcal{T}_n(i)$, 
$\lambda_{n,j}(G_{k,i}^{(n)})=0$.
\end{proof}

      \subsection{The coproduct on $\mathcal{H}$} \label{subsecdec}

In the previous subsection, we have constructed dual systems $(G_{k,i}^{(n)})_{\begin{subarray}{l} 0<k\leq n \\ i\in \mathcal{T}_k \end{subarray}} 
\subset\frac{\mathcal{F}_1}{\mathcal{F}_{n+1}}$ and $(\lambda_{k,i})_{\begin{subarray}{l} 0<k\leq n \\ i\in \mathcal{T}_k \end{subarray}}
\subset\frac{\mathcal{I}_n}{\mathcal{I}_0}$ that satisfy the following properties :
\begin{itemize}
\item $\lambda_{n,i}$ is a finite type invariant of degree $n$,
\item $\mathcal{T}_n=C_n\sqcup\mathcal{T}_n^\pi$, $\lambda_{n,i}$ is additive if $i\in C_n$, $\lambda_{n,i}$ is a product of some $\lambda_{k,i}$,  
  $k<n$, $i\in C_k$, if $i\in\mathcal{T}_n^\pi$,
\item $\displaystyle \frac{\mathcal{I}_n^c}{\mathcal{I}_{n-1}^c}=\prod_{i\in C_n} \mathbb{Q} \lambda_{n,i}$,
\item if $i\in C_n$, $G_{n,i}^{(n)}$ is multiplicative,
\item $G_{k,i}^{(n)}\in\mathcal{F}_k$, and, if $m \leq n$,  $G_{k,i}^{(n)}=G_{k,i}^{(m)}$ in $\mathcal{G}_m$.
\end{itemize}

\begin{proposition} \label{propinduction}
The family $(G_{k,i}^{(n)})_{\begin{subarray}{l} 0<k\leq n \\ i\in \mathcal{T}_k \end{subarray}}$ is a basis of 
$\frac{\mathcal{F}_1}{\mathcal{F}_{n+1}}$. The family $(\lambda_{k,i})_{\begin{subarray}{l} 0<k\leq n \\ i\in \mathcal{T}_k \end{subarray}}$ 
is the dual basis of $\frac{\mathcal{I}_n}{\mathcal{I}_0}$. Moreover:
$$\frac{\mathcal{I}_n}{\mathcal{I}_{n-1}}=\prod_{i\in \mathcal{T}_n} \mathbb{Q} \lambda_{n,i},
    \quad \frac{\mathcal{I}_n^\pi}{\mathcal{I}_{n-1}^\pi}=\prod_{i\in \mathcal{T}_n^\pi} \mathbb{Q} \lambda_{n,i},
    \quad \mathcal{G}_n=\bigoplus_{i\in \mathcal{T}_n} \mathbb{Q}\,G_{n,i}^{(n)}.$$
\end{proposition}
This result implies Proposition \ref{propadd}.

\begin{proof}
We will proceed by induction. For $n=1$, the result follows from Proposition \ref{propdeg1} and Corollary \ref{cordeg1}. 
Fix $n>1$. We will write $\meta\leq\me$ if $\eta_t\leq\varepsilon_t$ for all $t$, $\meta<\me$ if $\meta\leq\me$ and $\meta\neq\me$, 
and $\mo<\meta$ if $\eta_t>0$ for at least one $t$.
\begin{lemma}
 Consider $\lambda\in\mathcal{I}_n$ such that $\lambda(S^3)=0$. There are constants $\alpha_{m,\iota}$, for 
$1\leq m\leq n$ and $\iota\in\mathcal{T}_m^\pi$, such that:
$$\lambda(M_1\sharp M_2)=\lambda(M_1)+\lambda(M_2)+\sum_{m=1}^n\sum_{\iota=(\mk,\mi,\me)\in\mathcal{T}_m^\pi}
\alpha_{m,\iota} \sum_{\mo<\meta<\me}
 \prod_{1\leq t\leq \ell(\mk)} \binom{\varepsilon_t}{\eta_t} \lambda_{k_t,i_t}^{\eta_t}(M_1) \lambda_{k_t,i_t}^{\varepsilon_t-\eta_t}(M_2),$$
for all $\mathbb{Q}$HS's $M_1$ and $M_2$.
\end{lemma}
\paragraph{Remark}
The above expression of $\lambda(M_1\sharp M_2)$ defines a coproduct $\Delta$ on the algebra $\mathcal{H}$:
$$\Delta(\lambda)=\lambda\otimes 1+1\otimes\lambda+\sum_{m=1}^n\sum_{\iota=(\mk,\mi,\me)\in\mathcal{T}_m^\pi}
\alpha_{m,\iota} \sum_{\mo<\meta<\me}\prod_{1\leq t\leq \ell(\mk)} \binom{\varepsilon_t}{\eta_t} 
\lambda_{k_t,i_t}^{\eta_t}\otimes\lambda_{k_t,i_t}^{\varepsilon_t-\eta_t}.$$
Thus $\mathcal{H}$ has a Hopf algebra structure.
The primitive elements associated with this coproduct (the invariants $\lambda$ satisfying 
$\Delta(\lambda)=\lambda\otimes 1+1\otimes\lambda$) 
are the additive invariants. Milnor and Moore (\cite{MM}) proved that, under conditions, a Hopf algebra is generated 
as an algebra by its primitive elements. Here, we give an explicit and elementary proof of this result in our setting.

\begin{proof}
Define a bilinear map $\mu$ on $\mathcal{F}_0$ by
$$\mu(M_1,M_2)=\lambda(M_1\sharp M_2)-\lambda(M_1)-\lambda(M_2)$$ for all $\mathbb{Q}$HS's $M_1$ and $M_2$. Fix $M_2$, and 
consider $[M;(\frac{A_i'}{A_i})_{1\leq i\leq n}] \in\mathcal{F}_n$. We have:
\begin{eqnarray*}
 \mu([M;(\frac{A_i'}{A_i})_{1\leq i\leq n}],M_2) &=& \sum_{I\subset\{1,..,n\}}(-1)^{|I|} \mu(M((\frac{A_i'}{A_i})_{i\in I}),M_2) \\
 &=&\sum_{I\subset\{1,..,n\}}(-1)^{|I|}\lbp\lambda(M((\frac{A_i'}{A_i})_{i\in I})\sharp M_2)-\lambda(M((\frac{A_i'}{A_i})_{i\in I}))\rbp \\
 &=& -\lambda([M;(\frac{A_i'}{A_i})_{1\leq i\leq n},\frac{B_2}{B^3}]) \\
 &=& 0,
\end{eqnarray*}
where $B_2$ is a rational homology ball obtained from $M_2$ by removing an open ball.
Thus $\mu(.,M_2)$ is an invariant of degree at most $n-1$. Note that $\mu(S^3,M_2)=0$. 
By induction, $\mathcal{I}_{n-1}/\mathcal{I}_0$ is freely generated by the $\lambda_{k,i}$ for $0<k<n$ and $i\in \mathcal{T}_k$.
Hence we can write:
$$\mu(M_1,M_2)=\sum_{0<k<n}\sum_{i\in \mathcal{T}_k} \beta_{k,i}(M_2) \lambda_{k,i}(M_1).$$
Note that the sum may be infinite.
We have $\beta_{k,i}(M_2)=\mu(G_{k,i}^{(n)},M_2)$ and $\beta_{k,i}(S^3)=0$. Extend $\beta_{k,i}$ to $\mathcal{F}_0$ by linearity. 
Consider $[M;(\frac{A_i'}{A_i})_{i\in I}] \in\mathcal{F}_{n-k+1}$, 
$|I|=n-k+1$, and set $G_{k,i}^{(n)}=\sum_{u\in U} c_u [N_u;(\frac{B_j'}{B_j})_{j\in J_u}]$, where the $c_u$ are 
rational numbers and $|J_u|=k$ for all $u$.
\begin{eqnarray*}
 \beta_{k,i}([M;(\frac{A_i'}{A_i})_{i\in I}]) &=& 
 \sum_{u\in U} c_u \sum_{I'\subset I}\sum_{K_u\subset J_u}(-1)^{|I'|+|K_u|} 
  \mu(N_u((\frac{B_j'}{B_j})_{j\in K_u}),M((\frac{A_i'}{A_i})_{i\in I'}))\\
 &=& \sum_{u\in U} c_u \lambda([M\sharp N_u;(\frac{A_i'}{A_i})_{i\in I},(\frac{B_j'}{B_j})_{j\in J_u}]) \\
 &=& 0
\end{eqnarray*}
Thus $\beta_{k,i}$ is an invariant of degree at most $n-k$. Using the induction hypothesis, we can decompose the invariants 
$\beta_{k,i}$ and get:
$$\mu(M_1,M_2)=\sum_{m=1}^n\sum_{\iota=(\mk,\mi,\me)\in\mathcal{T}_m^\pi}\sum_{\mo<\meta<\me}
\alpha_{m,\iota}^{(\meta)} \prod_{1\leq t\leq \ell(\mk)} \binom{\varepsilon_t}{\eta_t} \lambda_{k_t,i_t}^{\eta_t}(M_1) 
\lambda_{k_t,i_t}^{\varepsilon_t-\eta_t}(M_2),$$
where the $\alpha_{m,\iota}^{(\meta)}$ are rational constants. 
It gives:
$$\lambda(M_1\sharp M_2)=
 \lambda(M_1)+\lambda(M_2)+\sum_{m=1}^n\sum_{\iota=(\mk,\mi,\me)\in\mathcal{T}_m^\pi}
\sum_{\mo<\meta<\me}\alpha_{m,\iota}^{(\meta)} \prod_{1\leq t\leq \ell(\mk)} \binom{\varepsilon_t}{\eta_t} \lambda_{k_t,i_t}^{\eta_t}(M_1) 
\lambda_{k_t,i_t}^{\varepsilon_t-\eta_t}(M_2). $$
Now, we use the commutativity and associativity of the connected sum to show that the well-determined constants 
$\alpha_{m,\iota}^{(\meta)}$ do not depend on $\meta$. 
The commutativity gives $\alpha_{m,\iota}^{(\me-\meta)} =\alpha_{m,\iota}^{(\meta)} $.
Consider $M_1=N_1\sharp N_2$.
\begin{eqnarray*}  && \hspace{-0.7cm} \lambda(N_1\sharp N_2\sharp M_2)=\\ && \lambda(N_1)+\lambda(N_2)+\lambda(M_2)
 +\sum_{m=1}^n\sum_{\iota=(\mk,\mi,\me)\in\mathcal{T}_m^\pi}\sum_{\mo<\meta<\me}\alpha_{m,\iota}^{(\meta)}  
\prod_{1\leq t\leq \ell(\mk)} \binom{\varepsilon_t}{\eta_t} \lambda_{k_t,i_t}^{\eta_t}(N_1) \lambda_{k_t,i_t}^{\varepsilon_t-\eta_t}(N_2)\\
 && \hspace{0.2cm} +\sum_{m=1}^n\sum_{\iota=(\mk,\mi,\me)\in\mathcal{T}_m^\pi}\sum_{\mo<\meta<\me}\alpha_{m,\iota}^{(\meta)}  
\sum_{\mo\leq\mnu\leq\meta}\prod_{1\leq t\leq \ell(\mk)} \binom{\varepsilon_t}{\eta_t} \binom{\eta_t}{\nu_t} \lambda_{k_t,i_t}^{\nu_t}(N_1) 
\lambda_{k_t,i_t}^{\eta_t-\nu_t}(N_2) \lambda_{k_t,i_t}^{\varepsilon_t-\eta_t}(M_2) \end{eqnarray*}
Consider $\mnu$ such that $\mo<\mnu\leq\meta$.
The terms $$\prod_{1\leq t\leq \ell(\mk)} \lambda_{k_t,i_t}^{\nu_t}(N_1) 
\lambda_{k_t,i_t}^{\eta_t-\nu_t}(N_2) \lambda_{k_t,i_t}^{\varepsilon_t-\eta_t}(M_2)$$ and   
$$\prod_{1\leq t\leq \ell(\mk)} \lambda_{k_t,i_t}^{\varepsilon_t-\eta_t}(N_1) 
\lambda_{k_t,i_t}^{\eta_t-\nu_t}(N_2) \lambda_{k_t,i_t}^{\nu_t}(M_2)$$ must have the same coefficient. 
Since $\binom{\varepsilon_t}{\eta_t} \binom{\eta_t}{\nu_t} = 
\binom{\varepsilon_t}{\varepsilon_t-\nu_t} \binom{\varepsilon_t-\nu_t}{\varepsilon_t-\eta_t}$, 
we have $\alpha_{m,\iota}^{\meta}=\alpha_{m,\iota}^{\me-\mnu}=\alpha_{m,\iota}^{\mnu}$. Now, consider any $\meta$ 
and $\mnu$ with $\mo<\meta,\mnu<\me$. Either there is $\mtau>\mo$ with $\mtau\leq\meta$ and $\mtau\leq\mnu$,
or we have $\meta\leq\me-\mnu$. In both cases, we get $\alpha_{m,\iota}^{\meta}=\alpha_{m,\iota}^{\mnu}$.
Finally:
$$\lambda(M_1\sharp M_2)=\lambda(M_1)+\lambda(M_2)+\sum_{m=1}^n\sum_{\iota=(\mk,\mi,\me)\in\mathcal{T}_m^\pi}
\alpha_{m,\iota} \sum_{\mo<\meta<\me}
 \prod_{1\leq t\leq \ell(\mk)} \binom{\varepsilon_t}{\eta_t} \lambda_{k_t,i_t}^{\eta_t}(M_1) \lambda_{k_t,i_t}^{\varepsilon_t-\eta_t}(M_2),$$
where $\alpha_{m,\iota}$ is the common value of the $\alpha_{m,\iota}^{(\meta)}$.
\end{proof}

Back to the proof of Proposition \ref{propinduction}, use the constants $\alpha_{m,\iota}$ given 
by the lemma to define an invariant $\tilde{\lambda}$:
$$\tilde{\lambda}=\lambda-\sum_{m=1}^n\sum_{\iota=(\mk,\mi,\me)\in\mathcal{T}_m^\pi}\alpha_{m,\iota}
\prod_{1\leq t\leq \ell(\mk)} \lambda_{k_t,i_t}^{\varepsilon_t}.$$
It is easy to see that $\tilde{\lambda}$ is additive. Thus $\lambda\in\mathcal{I}_n^c\oplus\mathcal{I}_n^\pi$, 
and $(\lambda_{k,i})_{\begin{subarray}{l} 0<k\leq n \\ i\in\mathcal{T}_k^\pi \end{subarray}}$ is a basis of 
$\mathcal{I}_n^\pi$.

It remains to show that $(G_{k,i}^{(n)})_{\begin{subarray}{l} 0<k\leq n \\ i\in \mathcal{T}_k \end{subarray}}$ is a basis of 
$\frac{\mathcal{F}_1}{\mathcal{F}_{n+1}}$. It suffices to show that $(G_{n,i}^{(n)})_{i\in \mathcal{T}_n}$ is a basis of $\mathcal{G}_n$. 
Consider $G\in\mathcal{G}_n$. We shall prove that the sum $\sum_{i\in \mathcal{T}_n}\lambda_{n,i}(G)G_{n,i}^{(n)}$ is finite 
and equal to $G$ in $\mathcal{G}_n$. The term $G$ is a finite linear combination of $\mathbb{Q}$HS's. 
Let $C_1(G)\subset C_1$ denote the set of all prime integers $p$ such that $\nu_p(M)\neq 0$ for a $\mathbb{Q}$HS $M$ in this 
combination. The set $C_1(G)$ is finite. If an invariant $\lambda_{n,i}$ is a multiple of an invariant $\nu_p$ 
for some $p\notin C_1(G)$, then $\lambda_{n,i}(G)=0$. Thus if $\lambda_{n,i}(G)\neq0$, then $\lambda_{n,i}$ is a product of invariants 
$\nu_p$ for $p\in C_1(G)$ and $\lambda_{k,j}$ for $1<k\leq n$ and $j\in C_k$. Recall the set $C_k$ is finite for all $k>1$. 
Hence the sum $\sum_{i\in \mathcal{T}_n}\lambda_{n,i}(G)G_{n,i}^{(n)}$ 
is well defined in $\mathcal{G}_n$, and is equal to $G$ since the $\lambda_{n,i}$ generate $\frac{\mathcal{I}_n}{\mathcal{I}_{n-1}}$.
\end{proof}

\begin{lemma} \label{lemma+}
 Let $M$ and $N$ be $\mathbb{Q}$HS's. For $n>0$:
$$\lbp (M-N)\in\mathcal{F}_{n+1} \rbp \Leftrightarrow \lbp Z_{k,KKT}(M-N)=0\textrm{ for all } k\leq \frac{n}{2} 
\textrm{ and } |H_1(M;\mathbb{Z})|\hspace{-0.12pt}=\hspace{-0.12pt}|H_1(N;\mathbb{Z})| \rbp.$$
\end{lemma}
\begin{proof}
 The direct implication is clear since the $Z_{k,KKT}$, $k\leq\frac{n}{2}$, and the $\nu_p$, $p$ prime, are finite type invariants of degree 
at most $n$. To see that Proposition \ref{propinduction} implies the converse implication, recall that the invariants $\lambda_{k,i}$, for 
$0<k\leq n$ and $i\in\mathcal{T}_k$, were defined in Subsection \ref{subsecdual} as products of linear combinations of the $\nu_p$, $p$ prime,  
and the $\lambda_{k,i}$, $0<k\leq n$, $i\in C_k$, that were defined from the $Z_{k,KKT}$, $k\leq\frac{n}{2}$ in Lemma \ref{lemmainvadd}.
\end{proof}
\proofof{Theorem \ref{th+}}
Since the LMO invariant is additive under connected sum, according to \cite{LMO}, and since Massuyeau proved that $Z_{LMO}$ 
satisfies the same splitting formulae as $Z_{KKT}$ in \cite{Mas}, the invariants $(\lambda_{n,i})_{i \in C_n}$ of Lemma \ref{lemmainvadd} 
could have been defined with $Z_{LMO}$ instead of $Z_{KKT}$. Therefore, Lemma \ref{lemma+} holds for $Z_{LMO}$ instead of $Z_{KKT}$ as well.
\makebox[\textwidth][r]{\fin}


\begin{thebibliography}{99}
\bibitem[AL]{AL} E. Auclair, C. Lescop, \emph{Clover calculus for homology 3-spheres via basic algebraic topology},
  Algebraic $\&$ Geometric Topology 5, p.71-106 (2005).
\bibitem[BN]{BN} D. Bar-Natan, \emph{On the Vassiliev knot invariants}, Topology 34 (2), p.423-472 (1995).
\bibitem[GGP]{GGP} S. Garoufalidis, M. Goussarov, M. Polyak, \emph{Calculus of clovers and finite type invariants of 3-manifolds},
  Geometry $\&$ Topology 5, p.75-108 (2001).
\bibitem[Hab]{Hab} K. Habiro, \emph{Claspers and finite type invariants of links}, Geometry $\&$ Topology 4, p.1-83 (2000).
\bibitem[KK]{KK} A. Kawauchi, S. Kojima, \emph{Algebraic classification of linking pairings on 3-manifolds},
 Mathematische Annalen 253, p.29-42 (1980).
\bibitem[Kon]{Kon} M. Kontsevich, \emph{Vassiliev's knots invariants}, Advances in Soviet Mathematics 16 (2), p.137-150 (1993).
\bibitem[KT]{KT} G. Kuperberg, D.P. Thurston, \emph{Perturbative 3-manifold invariants by cut-and-paste topology},
  arXiv:math/9912167v2 (2000).
\bibitem[Le]{Le} T.T.Q. Le, \emph{An invariant of integral homology 3-spheres which is universal for all finite type invariants},
  Solitons, geometry, and topology: on the crossroad, Amer. Math. Soc. Transl. Ser.2, 179, Amer. Math. Soc., Providence, RI, p.75-100 (1997).
\bibitem[LMO]{LMO} T.T.Q. Le, J. Murakami, T. Ohtsuki, \emph{On a universal perturbative invariant of 3-manifolds},
  Topology 37 (3), p.539-574 (1998).
\bibitem[Les]{Les} C. Lescop, \emph{Splitting formulae for the Kontsevich-Kuperberg-Thurston invariant of rational homology 3-spheres},
  Preprint arXiv:math/0411431v1 (2004).
\bibitem[Mas]{Mas} G. Massuyeau, \emph{Splitting formulas for the LMO invariant of rational homology three-spheres}, in preparation.
\bibitem[Mat]{Mat} S.V. Matveev, \emph{Generalized surgery of three-dimensional manifolds and representations
 of homology spheres}, Mathematical Notes of the Academy of Sciences of the USSR 42 (2), p.651-656 (1987).
\bibitem[Mey]{Mey} M.D. Meyerson, \emph{Representing homology classes of closed orientable surfaces}, 
  Proceedings of the American Mathematical Society 61, p.181-182 (1976).
\bibitem[MM]{MM} J.W. Milnor, J.C. Moore, \emph{On the structure of Hopf algebras}, Annals of Mathematics, 
  Ser.2, Vol.81, No2, p.211-264 (1965).
\bibitem[Mir]{Mir} R. Miranda, \emph{Nondegenerate symmetric bilinear forms on finite abelian 2-groups},
 Transactions of the American Mathematical Society 284 (2), p.535-542 (1983).
\bibitem[Wall]{Wall} C.T.C. Wall, \emph{Quadratic forms on finite groups, and related topics}, Topology 2, p.281-298 (1964).
\end{thebibliography}
\end{document}